\documentclass[12pt,reqno]{amsart}

\usepackage[all,cmtip]{xy}
\usepackage[dvipdfx,cmyk,table]{xcolor}
\usepackage{amsmath,amssymb}
\usepackage{amsthm}
\usepackage{amscd,latexsym}
\usepackage{verbatim}
\usepackage{mathrsfs}
\usepackage[all]{xy}
\usepackage{tikz}
 \usepackage{caption}
\usepackage[margin=2.5cm,marginpar=2cm]{geometry}

\usepackage[OT2,T1]{fontenc}

\usepackage{mathdots,url,hyperref}

\DeclareSymbolFont{cyrletters}{OT2}{wncyr}{m}{n}
\DeclareMathSymbol{\Sha}{\mathalpha}{cyrletters}{"58}

 \numberwithin{equation}{section}
\newcommand{\nc}{\newcommand}
\nc{\nt}{\newtheorem}
\nc{\dmo}{\DeclareMathOperator}
\nc{\enm}{\ensuremath}

\newtheorem{thm}{Theorem}

\newtheorem{prop}{Proposition}

\nt{claim}{Claim}
\newtheorem{lemma}{Lemma}
\newtheorem{remark}{Remark}
\nt{defn}{Definition}
\newtheorem{cor}{Corollary}
\nt{assumption}{Assumption}

\dmo{\Ind}{Ind}
\dmo{\cInd}{c-Ind}
\dmo{\Adj}{Ad}
\dmo{\PGL}{PGL}
\dmo{\SO}{SO}
\dmo{\Lie}{Lie}
\dmo{\Alt}{Alt}
\dmo{\reg}{reg}
\dmo{\sing}{sing}
\dmo{\supp}{supp}
\dmo{\tr}{tr}
\dmo{\Sym}{Sym}
\dmo{\Hom}{Hom}
\dmo{\Tor}{Tor}
\dmo{\Out}{Out}
\dmo{\Ht}{ht}
\dmo{\End}{End}
\dmo{\Mat}{Mat}
\dmo{\Tr}{Tr}
\dmo{\Isom}{Isom}
\dmo{\Span}{Span}
\dmo{\id}{id}

\dmo{\SL}{SL} \dmo{\sgn}{sgn} \dmo{\GL}{GL}  \dmo{\Mod}{mod} \dmo{\geo}{geo} \dmo{\re}{Re} \dmo{\Spec}{Spec}
\dmo{\Fr}{Fr} \dmo{\vol}{vol} \dmo{\Sets}{Sets} \dmo{\im}{im} \dmo{\diag}{diag} \dmo{\Ker}{ker} \dmo{\val}{val} \dmo{\ord}{ord}
\dmo{\Stab}{Stab} \dmo{\Ad}{Ad} \dmo{\rank}{rank} \dmo{\Symp}{Sp} \dmo{\Nm}{Nm} \dmo{\Norm}{Norm}\dmo{\Int}{Int}

\nc{\bb}{\mathbf}
\nc{\mb}{\mathbb}
\nc{\mf}{\mathfrak}

\nc{\bG}{\enm{{\mathbf G}}}
\nc{\hy}{H_{Y}}
\nc{\hyc}{H_{Y^{\perp}}}
\nc{\Aff}{\mathbb{A}}
\nc{\eps}{\varepsilon}

\nc{\ups}{\upsilon}
\nc{\GS}{\mathcal G \mathcal S}

\nc{\bks}{\enm{{\backslash}}}

 \nc{\isom}{\enm{{\overset{~}{\rightarrow}}}}

 \nc{\Z}{\enm{{\mathbb Z}}}

\nc{\Zp}{\enm{{\mathbb Z_p}}}

\nc{\Gm}{\enm{{\mathbb G_m}}}
\nc{\F}{\enm{{\mathbb F}}}
\nc{\Fp}{\enm{{\mathbb F}_p}}
\nc{\Fq}{\enm{{\mathbb F}_q}}
\nc{\Q}{\enm{{\mathbb Q}}}
\nc{\Qp}{\enm{{\mathbb Q_p}}}
\nc{\R}{\enm{{\mathbb R}}}
\nc{\N}{\enm{{\mathbb N}}}
\nc{\C}{\enm{{\mathbb C}}}
\nc{\CC}{\enm{{\mathcal C}}}
\nc{\half}{\enm{{\frac{1}{2}}}}
\nc{\BB}{\enm{{\mathcal B}}}
\nc{\flip}{\tilde{\eps}}

\nc{\ii}{\enm{{\mathcal I}}}
\nc{\jj}{\enm{{\mathcal J}}}
\nc{\OO}{\enm{{\mathcal O}}}
\nc{\f}{\enm{{\Sha}}}

\nc{\GGl}{\enm{{\mathfrak gl}}}
\nc{\GG}{\enm{{\mathfrak g}}}
\nc{\gd}{\enm{{\hat{\mathfrak g}}}}
\nc{\gm}{\enm{{\gamma}}}
\nc{\hh}{\enm{{\mathfrak h}}}

\nc{\II}{\enm{{\mathfrak a}}}
\nc{\LL}{\enm{{\mathfrak l}}}
\nc{\mm}{\enm{{\mathfrak m}}}
\nc{\pp}{\enm{{\mathfrak p}}}
\nc{\TT}{\enm{{\mathfrak t}}}
\nc{\Nc}{\enm{{\mathcal N}}}
\nc{\Cc}{\enm{{\mathcal C}}}
\nc{\HH}{\enm{{\mathfrak h}}}
\nc{\LS}{\enm{{\mathfrak s}}}
\nc{\iso}{\tilde{\rightarrow}}
\dmo{\res}{res}
 \nc{\Gd}{\enm{{\hat{G}}}}

  \nc{\Hd}{\enm{{\hat{H}}}}

\nc{\vt}{\enm{\vartheta}}
\nc{\lra}{\enm{\longrightarrow}}
\nc{\ra}{\enm{\rightarrow}}
\nc{\lip}{\enm{\langle}}
\nc{\rip}{\enm{\rangle}}

\nc{\nn}{\enm{\mathfrak n}}

\nc{\bsk}{\bigskip}

 \nc{\ol}{\overline}
\nc{\ul}{\underline}

\nc{\bH}{{\bf H}}
\nc{\bS}{{\bf S}}
\nc{\bT}{{\bf T}}
\nc{\bB}{{\bf B}}
\nc{\bA}{{\bf A}}
\nc{\bU}{{\bf U}}
\nc{\bN}{{\bf N}}
\nc{\bE}{{\bf E}}
\nc{\bX}{{\bf X}}
\nc{\bY}{{\bf Y}}
\nc{\bM}{{\bf M}}
\nc{\bW}{{\bf W}}

\nc{\beq}{\begin{equation*}}
\nc{\eeq}{\end{equation*}}
\nc{\mc}{\mathcal}

\nc{\xx}{\sts{\bowtie}}

\dmo{\elliptic}{ell}
\dmo{\disc}{disc}

 \long\def\sts#1{{{\color{magenta}#1}}}

\long\def\delete#1{}

\makeatletter
\def\Ddots{\mathinner{\mkern1mu\raise\p@
\vbox{\kern7\p@\hbox{.}}\mkern2mu
\raise4\p@\hbox{.}\mkern2mu\raise7\p@\hbox{.}\mkern1mu}}
\makeatother

\author{Arnab Mitra}
\author{Steven Spallone}

\address{Technion-Israel Institute of Technology, Haifa-3200003, Israel}
\email{00.arnab.mitra@gmail.com}
\address{Indian Institute of Science Education and Research, Pune-411021, India}
\email{sspallone@gmail.com}
\date{\today}
\begin{document}

\title{Towards a Goldberg-Shahidi pairing for classical groups}

\begin{abstract}
Let $G^{1}$ be an orthogonal, symplectic or unitary group over a local field and let $P=MN$ be a maximal parabolic subgroup. Then the Levi subgroup $M$ is the product of a group of the same type as $G^{1}$ and a general linear group, acting on vector spaces $X$ and $W$, respectively.  In this paper we decompose the unipotent radical $N$ of $P$ under the adjoint action of $M$, assuming $\dim W \leq \dim X$, excluding only the symplectic case with $\dim W$ odd.  The result is a Weyl-type integration formula for $N$ with applications to the theory of intertwining operators for parabolically induced representations of $G^1$.  Namely, one obtains a bilinear pairing on matrix coefficients, in the spirit of Goldberg-Shahidi, which detects the presence of poles of these operators at $0$.  
\end{abstract}

\maketitle
\tableofcontents

\section{Introduction}   
Let $G^1$ be a reductive group over a local field $F$, and $P=MN$ a maximal parabolic subgroup.  The Langlands-Shahidi method of defining $L$-functions draws our attention to the action of $M$ on $N$.  If $G^1$ is a classical group, to wit orthogonal, symplectic, or unitary, then $M$ is a product $G \times H$ of linear groups, with $G$ general linear and $H$ classical of the same kind as $G^1$. The interaction of $G$ and $H$ through the adjoint action of $M$ on $N$ provides information about representations of $G^1$ induced from $G \times H$.  In \cite{GS98} Goldberg and Shahidi define a pairing for matrix coefficients on $G$ and $H$, which detects the presence of a pole for the intertwining operator at $0$, in the case where $G$ and $H$ are matrix groups of the same size (i.e., they have the same number of rows and columns).  They then give an endoscopic interpretation of this pairing.  This work extends their theory to the case where the size of $H$ is greater than or equal to the size of $G$.  

Let $V$ be a finite-dimensional vector space over a field $F$ (with characteristic not equal to two), or over a quadratic extension $E$ of $F$, with respect to a nondegenerate symmetric, antisymmetric, or Hermitian form. Let us write $G^1$ for the isometry group of $V$.  A maximal parabolic subgroup $P$ of $G^1$ corresponds to an isotropic subspace $W$ of $V$.   We may decompose $P=MN$, where $N$ is the unipotent radical of $P$, and $M$ is a Levi component of $P$.  Then $M$ is isomorphic to a product $G \times H$ of groups, where $G=\GL(W)$ and $H$ is the isometry group of a nondegenerate subspace $X$ of $V$. Roughly speaking, our first goal is to provide a measure decomposition of $N$, or rather of a dense open subset $N_{\reg}$, under the adjoint action of $M$.   In this paper we assume that $\dim W=k \leq \dim X$, and that $k$ is even in case $V$ is symplectic.  
 
Goldberg and Shahidi define in \cite{GS98} a map which we write as $\Norm: N' \to H$, with $N'$ a certain open subset of $N$.  This map sends $\Int(M)$-orbits in $N'$ to conjugacy classes in $H$ and has properties that relate it to the Norm correspondence of Kottwitz and Shelstad \cite{KS}.  (Originally Shahidi \cite{S95} introduced this map in the even orthogonal case.)  The image is contained in the subset 
\beq
H_k= \{ h \in H \mid \rank(h-1; X) \leq k \}.
\eeq
  
To simplify this introduction, let us assume that $V$ is symplectic or orthogonal, and also that $k$ is even. 
Fix a nondegenerate subspace $Y$ of $X$ linearly isomorphic to $W$.  This is possible because we exclude the symplectic case with $k$ odd. Let $H_{Y}$ be the subgroup of elements of $H$ which fix the perpendicular space $Y^\perp$ pointwise.   By choosing an isomorphism $\xi$ from $Y$ to $W$, we can identify $H_{Y}$ with the fixed points $G^\theta$ of an involution $\theta$.   Further let $T$ be a maximal torus in $H_{Y}$. We write $T_{\reg}$ for the regular elements of $T$, and $H^{T}$ for the set $\Int(H)(T_{\reg})$. In Section \ref{densarg}, it is shown that $H^{T}$ is a nonempty Zariski open subset of $H_{k}$.
 
Given $\gm \in T_{\reg}$, we `match' it to a specified semisimple element $\gm_G \in G$ with the property that $\gm_G \cdot \theta(\gm_G) \in G^\theta$ (c.f. \cite{KS}).  A modification of $\gm_{G}$ gives an element $n_Y(\gm) \in N'$ so that $\Norm(n_Y(\gm))=\gm$.  The reader will find typical $n_Y(\gm)$ written out in matrix form in Section \ref{MatrixSection}.

We conjugate by $M$  and obtain a Zariski dense open subset
\begin{equation} \label{first_odd}
N_{\reg}= \bigcup_{Y,T} \{ \Int(m) n_{Y}(\gm) \mid m \in M, \gm \in T_{\reg} \}
\end{equation}
of $N$.  Here $Y$ runs over $H$-orbits of nondegenerate subspaces of $X$ which are linearly isomorphic to $W$, and $T$ runs over conjugacy classes of maximal tori in $H_{Y}$.

Next, suppose $F$ is a local field.  In Section \ref{Ad-invt}, we fix a standard $\Int(M)$-invariant measure $d_{M}n$ on $N$.  
Write $\Delta_{T}$ for the stabilizer in $M$ of $n_Y(\gm)$ for $\gm \in T_{\reg}$.  Computing the Jacobian of ``Shahidi's covering map''
\begin{equation*} 
\Sha_{T}: \left(M/ \Delta_{T} \right) \times T_{\reg}\to N
\end{equation*}
given by  $\Sha_{T}((m,\gm))=\Int(m)n_{Y}(\gm)$, we obtain our result:

\begin{thm} \label{intro_thm}
Suppose that $\dim W \leq \dim X$, that $V$ is symplectic, orthogonal, or unitary.  In the symplectic or orthogonal case, suppose that $\dim W$ is even.  Let $f \in L^{1}(N,d_Mn)$. Then we have:
\begin{equation*}
\int_N f(n) d_Mn=\sum_{Y,T}|W_{H_{Y}}(T)|^{-1}\int_{T} J_{T}(\gm)\int_{M/\Delta_{T}}f(\Int(m)n_{Y}(\gamma))\frac{dm}{dz}d\gamma.
\end{equation*}
Here $J_{T}(\gm)$ is a product of discriminant factors and a function
\beq
j_{T}(\gm)=|{\rm det}(\gm - 1;Y)|^{\half (\dim X-\dim W)}.
\eeq
\end{thm}

The function $J_{T}$ is defined precisely in Section \ref{finalform}.  As usual, $W_{H_{Y}}(T)$ denotes the Weyl group of $T$ in $H_{Y}$. The sum over $Y,T$ is the same as in Equation (\ref{first_odd}). There is a similar theorem in the case of $V$ orthogonal and $\dim W$ odd. (See Section \ref{finalform}). 
 
\begin{comment}
 \begin{equation*}
 \text{where } J_{T}(\gm)=|D_{H_{Y}}(\gm)|_{F}^\half |D^\theta_G(\gm_G)|_{F}^\half |{\rm det}_{F'}(\gm-1;Y)|_{F'}^{\half \dim_{F'} Y^{\perp}}.
 \end{equation*}
\end{thm}
Here $F'$ is $E$ in the unitary case and $F$ otherwise. Also, $D_{H_{Y}}$ is the usual discriminant, and $D^\theta_G$ is a $\theta$-twisted discriminant; see (\ref{disc.defn}) and (\ref{twist.disc.defn}) for their definitions. 
\end{comment}

In Section \ref{Residues.Section} we explain the role of these formulas in the theory of intertwining operators.  The theory involves a choice of matrix coefficients $f_G$ and $f_{H}$.   A `pairing' of the form
 
\beq
\GS(f_G,f_H) =  \sum_Y \sum_{ T_c \leq H_Y} |W_{H_Y}(T_c)|^{-1}\int_{T_c} j_{T_c}(\gm) I_{\gm}^H(f_H) I_{\gm_G}^G(f_G) d \gm,			 
\eeq
where $Y,T_c$ are as above, but with $T_c$ {\it elliptic}, should vanish if and only if the intertwining operator is holomorphic at $s=0$.  Here the distributions  $I_{\gm_G}^G$ and $I_{\gm}^H$ are orbital integrals.   

Our formulas generalize those of the second author in \cite{IFS} in the case of $\dim W=\dim X$.  There are obstacles not found in the equal-size case, since here the image of $\Norm$ lies in a subset of $H$ of measure $0$. The inspirational debt to Goldberg-Shahidi  (\cite{GS98}, \cite{GSIII}) should be obvious.

The theory developed in this paper is central to an ongoing project to study intertwining operators, which are given by integrals over unipotent radicals.  Here $F$ is $p$-adic. These intertwining operators are used to define Langlands-Shahidi $L$-functions, and the project connects the theory of $L$-functions to functoriality.  We refer the reader to the papers (\cite{S92},  \cite{S95}, \cite{GS98}, \cite{GS01}, \cite{GSIII}, \cite{Spallone},  \cite{Comp}, \cite{Cai-Xu}, \cite{Var}, \cite{WWL}, \cite{Yu1}, \cite{Yu2})   of Goldberg, Shahidi, Wen-Wei Li, Li Cai, Bin Xu, Xiaoxiang Yu, Varma and the second author for details and progress.

We emphasize that the best results thus far are only in the case of $\dim X=0$ (\cite{S92}) or $\dim W=\dim X$ (\cite{GS98}, \cite{Spallone}, \cite{Comp}, \cite{WWL}), but the results of this paper will open up many interesting cases with $\dim W < \dim X$.  The explicit Jacobian calculations at the end of this paper (as in \cite{IFS}) are crucial, because they allow for the application of endoscopic transfer to this project, as in \cite{WWL}.
 
For the case of quasisplit orthogonal and symplectic groups, the results in this paper up to Section \ref{densarg} are essentially in the work \cite{Yu2} of Yu and Wang.

We now delineate the sections of this paper.  After the preliminaries in Section \ref{Preliminaries}, we review the Goldberg-Shahidi Norm map in Section \ref{Norm_Corr}.  For a fixed nondegenerate subspace $Y$ of $X$ of dimension equal to that of $W$ (and an isomorphism $\xi_{Y}:Y \to W$), we define the group $H_{Y}$ and the map $\Xi$. Using these we obtain a section of the $\Norm$ map over $H_{Y}$. The map $\Sha$ is defined next in Section \ref{algebraic_th} and we compute its fibres and image. In Section \ref{MatrixSection} we display these objects explicitly as matrices when the reductive group $G^{1}$ is a split orthogonal, symplectic or a quasisplit unitary group. In Section \ref{densarg} we obtain a density argument for suitably regular elements, so as to obtain the full measure on $N$.  In Section \ref{Lie} we  provide a useful decomposition of the tangent space of $H$. 

We study the derivative of $\Sha$ in Section  \ref{local_now}, and calibrate differential forms on all the pieces.  Next in Section \ref{routes} we compute the Jacobian in the symplectic/orthogonal case, and in Section \ref{routes_2} we do the same in the unitary case. In Section \ref{LastSection} we consolidate our work, deduce the integration formula, and work out a couple of small rank cases.  

 Finally in Section \ref{Residues.Section} we exhibit a pairing of matrix coefficients expected to detect the presence of a pole for the intertwining operator at $0$.
 
{\bf Acknowledgments:} The authors would like to thank Sandeep Varma, Vivek Mallick, Amit Hogadi, Krishna Kaipa, and Freydoon Shahidi for useful conversations. Part of the work was done during the first author's Ph.D. and appears in his thesis. He would like to thank his advisor Dipendra Prasad for constant encouragement and support. This work was initiated during the second author's visit to the Tata Institute of Fundamental Research in Mumbai, and it is a pleasure to thank the institute for its support. During the second half of the project, the first author was supported by a postdoctoral fellowship funded by the Skirball Foundation via the Center for Advanced Studies in Mathematics at Ben-Gurion University of the Negev, and wishes to thank the foundation for its financial support. 

The authors would also like to thank the anonymous referee for a multitude of suggestions which greatly improved the paper.

{\bf Addendum:} During the review process of this manuscript, the independent paper \cite{Yu} of Xiaoxiang Yu appeared, which has some overlap with our work.  The overlap is an equivalent formula to our Theorem \ref{intro_thm}, for the case of quasisplit orthogonal and symplectic groups.

\section{Preliminaries} \label{Preliminaries}
\subsection{Notation}
Throughout this paper, we study orthogonal, symplectic, and unitary groups over a field $F$ such that the characteristic of $F$ is different from two.  From Section \ref{Norm_Corr} to Section \ref{Lie}, we will have no other restrictions on the field $F$. To set up the orthogonal and symplectic cases, let $V$ be a finite-dimensional $F$-vector space with a nondegenerate bilinear form $\Phi$, either symmetric or antisymmetric. To set up the unitary cases pick a quadratic extension $E$ of $F$, with nontrivial Galois automorphism $\sigma$.  Fix a nonzero element $\iota \in E$ of trace zero.  Let $V$ be a finite-dimensional vector space over $E$, and $\Phi$ a nondegenerate Hermitian form on $V$.  Thus $\Phi(a x,y)=\sigma(a) \Phi(x,y)=\Phi(x,\sigma(a)y)$ and $\Phi(y,x)=\sigma(\Phi(x,y))$ for $x,y \in V$, and $a \in E$. We will sometimes also write $\overline{x}$ for $\sigma(x)$. 
 
For a nonarchimedean local field $F$, $|\cdot|_{F}$ will denote the usual absolute value, i.e. normalized so that $|\varpi_F|_F=\frac{1}{q_F}$, with $\varpi_{F}$ a fixed uniformizer and $q_F$ the cardinality of the residue field. In particular, when $E/F$ is a quadratic Galois extension of local fields, $|x|_{E}= |N_{E/F}(x)|_{F}$ for an element $x\in E^{\times}$, and for $x \in F^{\times}$, $|x|_{F}$ is the positive square root of $|x|_{E}$. We may write $|\cdot|$ to mean $|\cdot|_{F}$ if no confusion is possible.

%Henceforth until Section \ref{sec_HC_cri}, when we speak of bases and dimensions of vector spaces, or calculate determinants of linear transformations, unless mentioned explicitly, we will consider them as $F$-spaces. For brevity of notation, we will change the convention slightly from Section \ref{sec_HC_cri} onwards which we mention explicitly in the beginning of that section.}

Write
\begin{equation*}
\Isom(V)=\{ g\in \GL(V)\ |\ \Phi(gv_{1},gv_{2})=\Phi(v_{1},v_{2})\ \forall v_{1},v_{2}\in V \}
\end{equation*}
for the group of linear isometries of $V$.

If $A$ is a subspace of $V$ we write $A^\perp$ for the set of vectors perpendicular to $A$.  We say that $A$ is nondegenerate if the restriction of $\Phi$ to $A$ is nondegenerate; in this case $A^\perp$ is also nondegenerate.

If $G$ is a group and $S$ is a subgroup, we write $Z_G(S)$ for the centralizer in $G$ of $S$, and $N_G(S)$ for the normalizer in $G$ of $S$. Given $g_0 \in G$, write $\Int(g_0)$ for the automorphism of $G$ given by $g \mapsto g_0 g g_0^{-1}$.

Our varieties are usually defined over $F$.  We use normal script (i.e. ``$G$''), resp. bold script (i.e. ``$\bG$'') for the $F$-points, resp. for the $\ol F$-points of these varieties.
If $\bG$ is an algebraic group we write $\bG^\circ$ for the identity component of $\bG$ in the Zariski topology.  

As usual, we use the Fraktur analogue of the Latin font to denote the Lie algebra of a given group. Thus $\GG$, $\hh$, etc. will denote the Lie algebras of $G$, $H$, etc. respectively.

\subsection{Differential forms and measures}\label{measnforms}
 
We recall here the definition of differential forms, their associated measures and some well known facts about them, when $F$ is a local field.  By ``manifold'' we mean an analytic finite-dimensional $F$-manifold in the sense of \cite{BourbDiff} or \cite{Serre}.  However, all the manifolds of interest to this paper will also be $F$-points of algebraic varieties, and the morphisms of interest to this paper are ``regular'' in the sense of algebraic geometry. If $X$ is a manifold and $p\in X$, we write $T_pX$ for the tangent space to $X$ at $p$.
 
\begin{defn} 
If $V$ is a vector space over $F$, write $\Alt^r(V)$ for the space of alternating forms on $r$-tuples of vectors in $V$.
\end{defn}
One denotes by $\underline{v}$ an $r$-tuple $(v_{1},...,v_{r}) \in V^r$.  
We will later employ the following construction from Chapter IX, Section 6 of \cite{BourbLie}:
\begin{defn}
Let
$$0\to U'\overset{i}\to U \overset{p}\to U'' \to 0$$
be an exact sequence of vector spaces, of dimensions $s$,$s+t$, and $t$, respectively. Let $\alpha''\in \Alt^{t}(U'')$ and $\alpha'\in\Alt^{s}(U')$. Then there is an alternating form $\alpha''\cap \alpha'\in\Alt^{s+t}(U)$, which is characterized by the following property: If $\underline{v}\in U^{t}$ and $\underline{v'}\in (U')^{s}$, then 
$$\alpha''\cap \alpha'(\underline{v},i(\underline{v'}))= \alpha''(p(\underline{v}))\alpha'(\underline{v'}).$$   
\end{defn}

\begin{defn}
Let $X$ be a manifold.  A (differential) $n$-form $\omega$ on $X$ is an analytic choice of alternating forms $\omega(p)\in \Alt^{n}(T_{p}X)$ for each point $p\in X$.
If $n=\dim X$, then an $n$-form is called a top form.
\end{defn}

Now suppose that $\omega$ is a top form, and put $n=\dim X$.  If $u^1,\ldots, u^n$ are coordinates on an open subset $U$ of $X$, then there is an analytic function $f$ on $U$ so that $\omega |_U$ is the form $f du^1 \wedge \cdots \wedge du^n$.  Then a real-valued measure $|\omega|$ on $X$ may be assembled by combining $|f|$ with the product of fixed Haar measures on the additive group of $F$ via the $u^i$.  (For details see \cite{BourbDiff}.)  In this case $|\omega|$ is called the measure associated to $\omega$.  If $G$ is a Lie group, this gives a one-to-one correspondence between left-invariant differential forms $\omega_G$ on $G$ (up to a nonzero constant in $F$) and left Haar measures $dg=|\omega_G|$ on $G$ (up to a positive constant in $\R$).
 
\begin{prop}\label{difffibre}  
Now let $X,Y$ be manifolds, and $\pi: X \to Y$ a proper surjective local diffeomorphism.  Let $\omega$ be a top form on $Y$ and $\pi^* (\omega)$ the pullback form on $X$. Suppose that $d\pi$ does not vanish at any point of $X$, and that the preimage of each point of $Y$ has precisely $d$ points. 
Let $f \in L^1(Y,|\omega|)$.  Then $f \circ \pi \in L^1(X,|\pi^{*}(\omega)|)$ and we have the identity
$$\int_{Y}f(y)|\omega|=\frac{1}{d}\int_{X}f(\pi(x))|\pi^{*}(\omega)|.$$
\end{prop}
 
\begin{proof} See Proposition 11 of Chapter V, Section 6 of \cite{Bourbint}.
 \end{proof}

\section{The Goldberg-Shahidi Norm}  \label{Norm_Corr}
\subsection{The unipotent radical}\label{pre_ur} 

Having fixed our ambient vector space $V$ equipped with a bilinear or Hermitian form $\Phi$ as above, we set $G^1=\Isom(V)$. Let $W$ be a totally isotropic subspace of $V$, and $P$ the stabilizer of $W$ in $G^1$.  Then $P$ is a maximal parabolic subgroup of $G^1$. Pick a subspace $W'$ of $V$ so that $W+W'$ is direct and nondegenerate. Let $X=(W+W')^{\perp}$. Let $M$ be the subgroup of $G^1$ that preserves $W$ and $W'$; it is a Levi subgroup of $P$.   Let $G=\GL(W)$ if $G^{1}$ is orthogonal or symplectic and $G={\rm Res}_{E/F}{\rm GL}_{E}(W)$ if $G^{1}$ is unitary. Let $H=\Isom(X)$. Note that when the form $\Phi$ is symmetric, the groups $G^{1}$ and $H$ are orthogonal groups and hence disconnected. Given $g \in G$ and $h \in H$, write $m=m(g,h)$ for the element in $M$ whose restriction to $W$ is $g$ and whose restriction to $X$ is $h$.  Then $m( \cdot ,\cdot )$ is an isomorphism from $G \times H$ to $M$.
 
Let $N$ be the unipotent radical of $P$.  An element $n\in N$ is determined by linear maps
$$\xi:X\to W, \ \xi':W'\to X, \ \eta:W'\to W$$
such that $n|_{W}={\rm Id}_{W}$,  $n|_{X}={\rm Id}_{X}+\xi$, and $n|_{W'}={\rm Id}_{W'}+\xi'+\eta$. Let us emphasize here that the maps $\xi$ and $\eta$ are $E$-linear in the case when $G^{1}$ is unitary. 
 
Define $\xi^{*}:W'\to X$ by $$\Phi(\xi^{*}(w'),x)=\Phi(w',\xi(x)),$$ for $x\in X$ and $w'\in W'$. Similarly define $\eta^{*}:W'\to W$, the adjoint of $\eta$, by $$\Phi(\eta^{*}(w_1'),w_2')=\Phi(w_1',\eta(w_2')),$$ for $w_1',w_2'\in W'$.   

The linear map $n$ determined by $\xi$, $\xi'$ and $\eta$ as above lies in $G^1$ if and only if both the following conditions are satisfied: 
\begin{enumerate}
\item $\xi^{*}+\xi'=0$
\item $\eta^{*}+\eta=\xi\xi'$. 
\end{enumerate}
Thus $\xi'$ is determined by $\xi$. Since $n$ is determined by $\xi$ and $\eta$, we write $n=n(\xi,\eta)$.  Thus for $n(\xi,\eta) \in N$ we have 
\begin{align}\label{defn}
\eta+\eta^{*}+\xi\xi^{*}=0.
\end{align}

Some useful calculations in what follows are $$n(\xi_{1},\eta_{1})n(\xi_{2},\eta_{2})=n(\xi_{1}+\xi_{2},\eta_{1}+\eta_{2}-\xi_{1}\xi_{2}^{*}), n(\xi,\eta)^{-1}=n(-\xi,\eta^*),$$ and $$\Int(m(g,h))(n)=mnm^{-1}=n(g\xi h^{-1}, g\eta g^{*}).$$

\subsection{Norm correspondence} \label{w_0.here}
Write $N'$ for the variety of matrices $n(\xi, \eta) \in N$ with $\eta$ invertible. 
\begin{defn}
Suppose that $n=n(\xi,\eta) \in N'$.  Let $\Norm(n):X\to X$ be the linear transformation given by 
\begin{equation*}
\Norm(n)=1+\xi^{*}\eta^{-1}\xi.
\end{equation*}
\end{defn}

The above definition originates from Shahidi and Goldberg-Shahidi who studied this map in connection with the theory of intertwining operators (\cite{S95}, \cite{GS98}). To explain, 
set $\varepsilon$ to be $(-1)^{\dim_{F}(W)}$ in the symplectic and orthogonal cases and to be $(-1)^{\dim_{E}(W)}$ in the unitary case.
Also fix a isomorphism $\ups_{0}: W \to W'$ with $\ups_{0}^*=\varepsilon \ups_{0}$; this amounts to sending a basis for $W$ to $\varepsilon$ times its dual basis in $W'$.  Let $w_0$ be the transformation of $V$ which is multiplication by $\varepsilon$ on $X$, given by $\ups_{0}^*$ on $W$, and by $\ups_{0}^{-1}$ on $W'$.  Note that $w_0 \in G^1$, and $w_0^2 \in Z(M)$.

Let $N^-$ be the subgroup of elements of $G^1$ which restrict to the identity on the three spaces $W'$, $(X+W')/W'$, and $V/ (X+W')$.  This is the unipotent radical of the parabolic subgroup with Levi subgroup $M$ opposite to $P$.  As with $n \in N$, an element $n^- \in N^-$ is determined by maps $\xi^-: X \to W'$ and $ \eta^-: W \to W'$ satisfying $\eta^- + (\eta^-)^* + \xi^- (\xi^-)^*=0$.  We then write $n^-=n(\xi^-, \eta^-)$ in this case. 

Intertwining operators in this context are, briefly, operators of the form 
\begin{equation} \label{int.op}
f \mapsto \int_N f(w_0^{-1}n -) dn
\end{equation}
for functions $f$ on $G^1$ with a certain left $P$-invariance.  One rewrites the argument of $f$ to take advantage of this invariance, with the following lemma:
 
\begin{lemma} \label{ralph}
Let $n=n(\xi,\eta)$.  Then $w_0^{-1} n \in P N^{-}$ if and only if $n \in N'$, in which case 
\begin{equation} \label{etaisom}
w_0^{-1} n=  m(\varepsilon \ups_{0}^{-1} \eta^{-*} , \varepsilon \cdot \Norm(n))  \cdot   n(-\eta^* \eta^{-1} \xi, \eta^*) \cdot  n^-(\eta^{-1} \xi,  \eta^{-1}).
\end{equation} 
\end{lemma}

\begin{proof}
Suppose $w_0^{-1} n(\xi,\eta) =p n^-$, with $p=mn \in MN$ and $ n^- \in  N^-$.  Then on $W'$ we have $\ups_{0} \eta =m|_{W'}$, thus $\eta$ is an isomorphism. Conversely, if $\eta$ is an isomorphism, then one explicitly multiplies out either side of (\ref{etaisom}) and verifies their equality.
\end{proof}

Thus the origin of $\Norm(n)$.  While we are here, let us extract a simple corollary from this calculation:
\begin{cor} 
If $G^{1}$ is a symplectic or an orthogonal group, then we have $\det(\Norm(n)) = \varepsilon$.
\qed
\end{cor} 

Here are some straightforward properties of $\Norm$:

\begin {lemma}\label{basicnorm}
\begin{enumerate}
\item $\Norm(n)\in H$, and $\rank(\Norm(n)-1; X) \leq \dim W$. (Equality is attained if and only if $\xi$ is surjective.) 
\item If $n \in N'$ and $m(g,h) \in M$, then $\Norm(\Int(m(g,h))n)=\Int(h)\Norm(n)$.
\item If $g\in G$, then the pair $(g\xi,g\eta g^{*})$ also satisfies (\ref{defn}) and 
$$\Norm (n(\xi,\eta))=\Norm(n(g\xi,g\eta g^{*})).$$ 
\end{enumerate}
\qed
\end{lemma}

\subsection{Sections of $\Norm$}\label{sect_norm}

The $\Norm$ map is generally far from being surjective.  If $\dim W < \dim X$, then any $\xi$ as above will have a nontrivial kernel, and therefore any $h=\Norm(n)$ will have a fixed-point space $X^h$ of dimension at least $\dim X-\dim W$. The restrictions of $h$ to $X^h$ and $(X^h)^\perp$ are isometries.

Conversely, let $Y$ be a nondegenerate subspace of $X$ linearly isomorphic to $W$. Next we will define a section of the $\Norm$ map on an open subset of the isometry group of $Y$.  (This subset will be nonempty except in the symplectic case when $\dim W$ is odd.)  As earlier let $\dim X=m$ and $\dim W=k$.

\subsection{Nondegenerate subspaces of $X$}\label{defy}

 We begin with the following standard result whose proof we leave to the reader.
 \begin{lemma} \label{iamalemma}
Let $(X,\Phi)$ be a vector space with a nondegenerate bilinear (or sesquilinear) form and $\gm$ a semisimple isometry.  Then the fixed point space $X^\gm$ is nondegenerate.
\qed
\end{lemma}

\begin{prop}\label{scrsk} 
$H$ naturally acts  on the set of nondegenerate subspaces of $X$ of dimension $k$.
Write
\beq
\mc Y_k= H \bks \{Y\subseteq X\mid \dim Y=k \textrm{ and }Y\ \textrm{is nondegenerate}\} 
\eeq
for the set of orbits of this $H$-action.
 \begin{enumerate}
\item The set $\mc Y_k$ is empty if and only if $X$ is symplectic and $k$ is odd.
\item If $F$ is algebraically closed, and $\mc Y_k \neq \emptyset$, then $\mc Y_k$ is a singleton.
\end{enumerate}
\end{prop}

\begin{proof}  
The first statement is clear if $X$ is symplectic.  If $X$ is not symplectic, then via Lemma 1 in Section 6, no. 1 of \cite{BourbQuad}, there is a nonisotropic vector in $X$, and one can thus build up the required $Y$.  

The second statement is proved on a case-by-case basis.  In the symplectic case, the equivalence class of nondegenerate subspaces under the standard $H$-action is determined by its dimension. In the orthogonal case, since the ground field is algebraically closed, any two symmetric bilinear forms are equivalent. Hence such a statement holds true in this case as well. The statement is vacuous in the unitary case.
\end{proof}

\begin{remark}
Of course, if $X$ is symplectic and $k$ is even, then $\mc Y_k$ is a singleton.  Theorem 63:20 of \cite{O'Meara} and Theorem 1.1 (ii) of \cite{Scharlau} classify quadratic and hermitian forms of a given dimension when $F$ is a local field.   Calculation of $\mc Y_k$ is then essentially an application of Witt's theorem, and the classical invariant theory of these forms.  
\end{remark}
 
Let $h=\Norm(n)$ be semisimple with $\dim X^{h}=m-k$. In Section \ref{densarg} we will see that this is the case for most $n$, assuming that we exclude the case where $k$ is odd and $\Phi$ is symplectic. Henceforth we will always assume that $k$ is even if $\Phi$ is symplectic. Define $Y=(X^{h})^{\perp}$. By Lemma \ref{iamalemma} above, this space is nondegenerate of dimension $k$. 

\begin{defn} For a nondegenerate subspace $Y$ of $X$, put
\beq
H_Y=\left \{ h \in H; h|_{Y^\perp}=\id_{Y^\perp} \right \}.
\eeq
\end{defn}
Note that $\Isom(Y)$ is isomorphic to $H_Y$; one extends an isometry of $Y$ to one of $X$ by requiring it to be the identity on $Y^\perp$.   We have $\hy \cap H_{Y^{\perp}}=\{\id_X \}$ and $\Stab_{H}(Y)=\Stab_{H}(Y^{\perp})=\hy \cdot H_{Y^{\perp}}$.  This is a direct product in $H$, which we will write as $\hy \times H_{Y^{\perp}}$.

\subsection{Relating $X$ to $W$}\label{relxw}
Fix an isomorphism $\xi=\xi_{Y}:Y \xrightarrow{\sim }W$. (We emphasize that in the case when $H$ is a unitary group we require $\xi$ to be an $E$-isomorphism.) Extend it to $X$ by defining it to be 0 over $Y^{\perp}$. We collect some properties of $\xi$.  

\begin{lemma}\label{kerxit} \label{invert}
\begin{enumerate}
\item $\xi^*: W' \to X$ is injective with image $Y$.
\item $\xi \xi^*: W' \to W$ is an isomorphism.
\end{enumerate} 
\end{lemma}
\begin{proof}
 The first statement is immediate.  For the second, note that $(\xi|_{Y})^{*}\in \Hom(W',Y)$ is an isomorphism and that $\xi \xi^*=(\xi|_{Y}) (\xi|_{Y})^*$.

\end{proof}

\begin{defn} 
Define $\ups=\ups_{Y}=(\xi\xi^{*})^{-1}: W \to W'$ and $\xi^+= \xi^* \ups=\xi^{*}(\xi\xi^{*})^{-1}$.
\end{defn}

 Thus $\xi^{+}$ is a right inverse of $\xi$.  It is easy to see that $\xi^+ \xi=P_{Y}$ (where $P_{Y}$ is the projection to the space $Y$). 
 
\begin{defn}\label{xidef}
Define a map $\Xi: \End(X) \to \End(W)$ by $\Xi(A)=\xi A\xi^+$. 

%Also define $\Xi^{+}: \End(W) \to \End(X)$ by $\Xi^{+}(A)=\xi^{+} A\xi$. 
\end{defn}
Note that $\Xi$ can equivalently be defined as the composite of 
\[
\End(X) \to \End(Y) \to \End(W),
\] 
where the first map is $A\mapsto P_{Y}AP_{Y}$ and the second is conjugation by $\xi|_{Y}$. We will occasionally write ${}^\xi A$ instead of $\Xi(A)$ for brevity. We now have the following simple proposition.

\begin{prop}\label{propstab}
\begin{enumerate}
\item Let $A,B \in \End(X)$.  If $A$ or $B$ commutes with $P_{Y}$, then $\Xi(AB)=\Xi(A)\Xi(B)$.
\item Let $A \in \End(X)$ so that $A$ commutes with $P_{Y}$ and $A|_{Y}$ is invertible.  Then $\Xi(A)$ is invertible.
\item $\Xi|_{\Stab_H(Y)}$ is a group homomorphism from $\Stab_H(Y)$ to $G$ with kernel $\hyc$. 
\item $\Xi$ restricts to an injection from $\hy$ into $G$.
\end{enumerate}
\qed
\end{prop}

Next we define a nondegenerate bilinear form $\Psi_{W}$ on $W$ via
\begin{defn}
\begin{equation*}
\begin{split}
\Psi_{W}(w_1,w_{2}) &=\Phi(\xi^+ w_1,\xi^+w_2) \\	
					&= \Phi(w_1, \upsilon w_2),\\
\end{split}
\end{equation*}
for $w_1,w_2 \in W$.
\end{defn}
% Evidently, the form $\Psi_W$ is nondegenerate. 

\begin{defn}
Write $\tau=\tau_Y$ for the adjoint map for the space $(W,\Psi_{W})$. Explicitly $\tau: \End(W) \to \End(W)$ is the antiinvolution such that
$$\tau(A)=\upsilon^{-1}A^{*}\upsilon.$$
Let $\theta=\theta_Y: G \to G$ be the involution $\theta(g)=\tau(g)^{-1}$.
\end{defn}
%We have $\Psi_{W}(gw_{1},w_{2})=\Psi_{W}(w_{1},\tau(g)w_{2})$. 
%We will identify it as an isometry group of $W$ as follows.  
 
Let $G^{\theta}$ be the group of fixed points of $G$ under the involution $\theta$. The next lemma is straightforward from definitions.
\begin{lemma}\label{gtstruc}
\begin{enumerate}
\item For all $A \in \End(X)$, we have $\tau(\Xi(A))=\Xi(A^*)$.
%\item The form $\Psi_W$ is nondegenerate and $\Psi_{W}(gw_{1},w_{2})=\Psi_{W}(w_{1},\tau(g)w_{2})$.
\item The map $\Xi$ induces an isomorphism from $H_{Y}$ to $G^{\theta}$, which is equal to $\Isom(W,\Psi_W)$.
\end{enumerate}
\qed
\end{lemma}
 
 Let $T$ be a maximal torus in $H_{Y}$. Then ${}^\xi T=\Xi(T)$ is a maximal torus in $G^{\theta}$.  We define $T_{G}=Z_{G}({}^\xi T)$.  Note that
\begin{equation*}
 {}^\xi T=(T_{G} \cap G^{\theta})^{\circ}.
\end{equation*}

\subsection{A section of $\Norm$ over $\hy$}

\begin{defn}  
Let $h \in \Stab_H(Y)$ with $h-1|_{Y}$ invertible.  
Define 
\beq
h_{G}=\left( \Xi(h-1) \right)^{-1} \in G.
\eeq
\end{defn}

\begin{prop} \label{About_h_G}
Let $h \in \Stab_H(Y)$ with $h-1|_{Y}$ invertible.  We have
\begin{enumerate}
\item $1+h_{G}+\tau(h_{G})=0$.
\item $h_{G} \cdot \theta(h_{G})=-\Xi(h^{-1})$.
\item If $h$ centralizes $T$, then $h_G \in T_G$.
\end{enumerate}
\end{prop}
\begin{proof}
For the first statement, multiply the expression
\begin{equation*}
1+(\Xi(h-1))^{-1}+(\Xi(h^{-1}-1))^{-1}=0
\end{equation*}
by $\Xi(h-1)\Xi(h^{-1}-1)$ to obtain
\begin{equation*}
\Xi(h-1)\Xi(h^{-1}-1)+\Xi(h^{-1}-1)+\Xi(h-1)=0.
\end{equation*}
The second and the third statements are straightforward from the definitions.
\end{proof}

\begin{remark}
This property in Proposition \ref{About_h_G} (ii) points to the usage of the term ``$\Norm$'' here.
\end{remark}

We can now construct the section that we want.

\begin{defn}
For $h \in \Stab_H(Y)$ with $h-1|_{Y}$ invertible, let
\begin{equation*}
\begin{split}
n_{Y}(h) &=n(\xi_{Y},h_G \ups^{-1} ) \\
               &=n(\xi_{Y}, \xi_{Y}(h|_{Y}-1)^{-1}\xi_{Y}^{*} ) \\
		&= n(\xi_Y, (\xi_{Y}(h-1)\xi_{Y}^+)^{-1} \xi_{Y} \xi_{Y}^{*}). \\
                 \end{split}
\end{equation*}
\end{defn}

Note that the pair $(\xi_{Y},h_G \ups^{-1})$ satisfies (\ref{defn}) by Proposition \ref{About_h_G}.

\begin{prop} \label{section}
Let $h \in \hy$ with $h-1|_{Y}$ invertible. Then $\Norm(n_{Y}(h))=h$.
\end{prop}

\begin{proof}
We compute
\begin{equation*}
\begin{split}
\Norm(n_Y(h)) &= 1+\xi_{Y}^* (h_G \ups^{-1})^{-1} \xi_{Y} \\
			&=1+ P_{Y} h P_{Y}-P_{Y} \\
			&= h, \\
\end{split}
\end{equation*}
as desired.  
\end{proof}

\section{Shahidi's covering map $\Sha_T$}\label{algebraic_th}
Proposition \ref{section} allows us to chose elements $n\in N'$ such that $\Norm(n)$ is semisimple. Fix one such $n$. Let $Y=X^{\Norm(n)}$. Suppose $\dim Y=\dim W$. As in Section \ref{relxw}, fix $\xi_{Y}:X\to W$ with kernel $Y^\perp$.  Let $T$ be a maximal torus in $H_{Y}$. 

Observe that if $w\in H$ normalizes $T$, then it stabilizes the space $Y$. The following is immediate:
\begin{lemma}\label{basic xi prop}
\begin{enumerate}
\item $N_{H}(T)=N_{H_{Y}}(T)\times H_{Y^{\perp}}$.
\item $Z_{H}(T)=Z_{H_Y}(T) \times H_{Y^{\perp}}$.
\end{enumerate}
\qed
\end{lemma}  
\begin{defn}
 Define $\Delta: \Stab_H(Y) \to M$ via $\Delta(h)=m(\Xi( h), h)$.  Write $\Delta_{T}$ for the image of $Z_{H}(T)$ under $\Delta$.
\end{defn}

\begin{defn}
Define
\begin{equation}\label{nylower}
N_{Y}= \{ n \in N' \mid X^{\Norm(n)}= Y^{\perp} \}.
\end{equation}
\end{defn}
Now we delve deeper into the structure of $N'$. As in \cite{IFS}, we bifurcate into two cases.

\subsection{$H_{Y}$ is not odd orthogonal} \label{alg_theory}
 
\begin{defn}
Let $T_{\reg}$ be the set of those regular elements $\gm \in T$, for which $(\gm-1)|_Y$ is invertible.
\end{defn}

Thus $n_Y$ restricts to a section $n_Y: T_{\reg} \to N'$ of $\Norm$.

\begin{defn} We define subsets of $N'$ via
\beq
N_{Y,T}=\{n\in N' \mid  \ \Norm(n) {\rm \ is \ conjugate \ in}\ H_{Y}{\rm \ to \ an\  element\ of } \ T_{\rm reg}\ \},
\eeq
\beq 
N^{Y,T}=\Int(H)N_{Y,T}.
\eeq
\end{defn}
Recall that $H_{Y}$ is the $F$-points of the classical group given by the linear isometries of the form $\Phi|_{Y}$. Note that $N_{Y,T}\subseteq N_{Y}$. The above definition is motivated by the following proposition:
 \begin{prop}\label{decomny}
 Let $h_{1},h_{2}\in H$. Then $\Int(h_{1})(N_{Y,T})\cap \Int(h_{2})(N_{Y,T})$ is nonempty if and only if $h_{2}^{-1}h_{1}\in \Stab_{H}(Y)$, in which case 
$$\Int(h_{1})(N_{Y,T})= \Int(h_{2})(N_{Y,T}).$$
\end{prop}
\begin{proof}
Let $h_{1}, h_{2}\in H$ and $ h\in \Stab_{H}(Y)$. A simple calculation using Lemma \ref{basicnorm} shows that $\Int(h)$ preserves $N_{Y,T}$, and hence $\Int(h_{1})(N_{Y,T})=\Int(h_{2})(N_{Y,T})$ if $h_{2}^{-1}h_{1}\in \Stab_{H}(Y)$. To prove the converse, let $n_{1},n_{2}\in N_{Y,T}$ such that $\Int(h_{1})(n_{1})=\Int(h_{2})(n_{2})$. Thus $\Int(h_{2}^{-1}h_{1})(n_{1})=n_{2}$ and $X^{\Norm(\Int(h_{2}^{-1}h_{1})(n_{1}))}=X^{\Norm(n_{2})}=Y^{\perp}$. Note that for any $h'\in H$, $X^{\Norm(\Int(h')(n))}=h'(Y^{\perp})$ which implies that $h_{2}^{-1}h_{1}\in \Stab_{H}(Y^{\perp})=\Stab_{H}(Y)$.
\end{proof}

\begin{defn}
Further define the map $\Sha_{T}:  M \times T_{\rm reg}\to N'$ by
\begin{equation} \label{fake_Sha}
\begin{split}
\Sha_{T}((m(g,h),\gamma)) &= \Int(m(g,h))n_{Y}(\gamma)\\
							%&= n(g\xi_{Y} h^{-1},g\eta_{Y}(\gamma)g^{*})\\
							&=n(g \xi_{Y} h^{-1}, g \gamma_G \ups^{-1} g^{*}). \\
\end{split}
\end{equation}
\end{defn}

We occasionally drop the subscript ``$Y$'' from  $\xi_{Y}$ and $n_{Y}$ when there is no scope of confusion. Our $\Sha_T$ is a version of maps found in the work (\cite{S92} ,\cite{S95}, \cite{GS98}, \cite{GS01}, \cite{GSIII}) of Shahidi and Goldberg-Shahidi.

\begin{prop}\label{surjprop}  We have 
\begin{enumerate}
\item $\Norm(\Sha_{T}((m(g,h),\gamma)))=  h \gamma h^{-1}$.
\item The image of $\Sha_{T}((G\times H_{Y}) \times T_{\reg})$ equals $N_{Y,T}$.
\item The map $\Sha_T$ surjects onto $N^{Y,T}$.
\end{enumerate}
\end{prop}

\begin{proof}
The first statement follows from Lemma \ref{basicnorm}(ii) and Proposition \ref{section}. The third statement follows from the second statement. 

We now prove the second statement. The fact that the image of $\Sha_{T}$ is contained in $N_{Y,T}$ follows from the first statement. For the other direction, let $n=n(\xi_{1},\eta_{1})\in N_{Y,T}$, so that we can choose elements $h\in H_{Y}$ and $\gm \in T_{\rm reg}$ be such that $h\Norm(n)h^{-1}=\gm$. Replacing $n$ by $\Int(m(1,h^{-1}))(n)$, we can suppose that $h=1$. Therefore we have
\[
\gm-1=\xi_{1}^{*}\eta_{1}^{-1}\xi_{1}.
\]
As $\gm \in T_{\reg}$, the map $\xi_{1}^{*}\eta_{1}^{-1}\xi_{1}|_{Y}$, and thus $\xi_{1}|_{Y}$, is invertible. Define $g\in G$ such that $g=\xi_{1}(\xi|_{Y})^{-1}$. Since we also have $\gamma=\Norm(\xi,\gm_G \ups^{-1})$, we obtain that $\eta_{1}=g\gm_{G}\ups^{-1}g^{*}$. Clearly $\Sha_{T}(m(g,1),\gm)=n(\xi_{1},\eta_{1})$, and therefore $N_{Y,T}$ is contained in $\Sha_{T}((G\times H_{Y})\times T_{\reg})$.  
\end{proof}

The following is a simple calculation:

\begin{prop} \label{Sha.invariance}  Let $h \in \Stab_H(Y)$, $m \in M$, and $\gm \in T_{\reg}$.  Then 
\beq
\Sha_{h^{-1}Th} ((m \Delta(h), h^{-1} \gm h))=\Sha_T((m , \gm)).
\eeq
\qed
\end{prop}

(Note that $h^{-1}Th$ is again a maximal torus in $H_Y$.)

Consider the action of $N_{H}(T)$ on $M \times T_{{\rm reg}}$ given by:
\begin{equation*}
w : ((m ,\gm)) \mapsto (m \Delta(w) , w^{-1} \gm w),
\end{equation*}
for $w \in N_{H}(T)$.  

\begin{prop} 
 Let $m \in M$ and $\gm \in T_{\rm reg}$. 
 Then the fibre of $\Sha_{T}$ containing $(m,\gm)$ is equal to
\begin{equation}\label{prop_fiber}
\{ (m \Delta( w) ,  w^{-1} \gm  w) \mid   w \in N_{H}(T) \}.
\end{equation}
\end{prop}
  
\begin{proof}  
Using the $M$-equivariance of the $\Sha_{T}$, without loss of generality, we can assume that $m=1$. The fact that the elements in (\ref{prop_fiber}) are in the fibre follows from Proposition \ref{Sha.invariance}. Suppose now that there are $g\in G$, $h \in H$, and $\gm, \gm' \in T_{{\rm reg}}$ so that
\begin{equation} \label{labeled}
\Sha_{T}((m(g, h), \gm'))=\Sha_{T}((1 , \gm)).
\end{equation}
 Thus $n(g\xi h^{-1},g(\gm')_G \ups^{-1}g^{*})=n(\xi, \gm_G \ups^{-1} )$. Comparing the first coordinates we find $g \xi =\xi h$ and therefore $g=\Xi(h)$.

Taking $\Norm$ of both sides of (\ref{labeled}) gives
\begin{equation*}
h \gm' h^{-1}=\gm
\end{equation*}
by Proposition \ref{surjprop}.  It remains to show that $h\in N_{H}(T)$. Since $\gm, \gm' \in T_{{\rm reg}}$, noting $\ker (\gm-1)=\ker(\gm'-1)=Y^{\perp}$, it follows that $h$ stabilizes $Y^{\perp}$. Writing $h=h_{Y}h_{Y^{\perp}}$, where $h_{Y}\in H_{Y}$ and $h_{Y^{\perp}}\in H_{Y^{\perp}}$, it follows that $h_{Y}\in N_{H_{Y}}(T)$. The result now follows from Lemma \ref{basic xi prop}.
\end{proof} 
 
\begin{cor} 
The map $\Sha_{T}$ descends to a surjective map  
\begin{equation} \label{real_Fy}
\Sha_{T}: M/\Delta_{T} \times T_{{\rm reg}}\to N^{Y,T}.
\end{equation}
The fibres of $\Sha_{T}$ are the same as the $W_{H_{Y}}(T)$-orbits on $M/\Delta_{T} \times T_{\rm reg}$.  In other words, the fibre of $\Sha_{T}$ containing $(m, \gm)$ is equal to
\begin{equation*}
\{ (m \Delta( w) ,  w^{-1} \gm  w) \mid  w \in W_{H_{Y}}(T) \}.
\end{equation*}
\qed
\end{cor}

\subsection{$H_{Y}$ is odd orthogonal}
The difficulty in the orthogonal case with $Y$ odd-dimensional is that every element $\gm \in T$ has $1$ as an eigenvalue, so $n_Y(\gm)$ is not well defined. We remedy this by first multiplying $\gm$ by $-1$ ``on $Y$'', and then applying $n_Y$.  Let $\epsilon_{Y}=P_{Y^{\perp}}-P_Y \in H_{Y}$.

\begin{defn}
Let $T_{\reg}$ be the set of regular elements $\gm \in T$ so that $(\epsilon_Y\gm -1)|_Y$ is invertible.
\end{defn}

Thus we have a section $n_Y: \epsilon_Y T_{\reg} \to N'$ of $\Norm$.

 Analogous to the previous case we define the following objects:
\begin{defn}
Define
$$N_{Y,T}=\{n\in N' \mid  \Norm(n) {\rm \ is \ conjugate \ in}\ H_{Y}{\rm \ to \ an\  element\ of } \ \epsilon_Y T_{{\rm reg}}   \},$$
$$N^{Y,T}= \Int(H)(N_{Y,T}).$$
\end{defn}
%(The same arguments demonstrate the veracity of Proposition \ref{decomny} in this case as well.)

\begin{defn}
Define $\Sha_{T} : M \times T_{\rm reg}\to N'$ by
\begin{equation}
\Sha_{T}((m , \gamma)) =\Int(m) n_Y(\epsilon_Y \gm).
\end{equation} 
\end{defn}
As in Proposition \ref{surjprop} we have:
\begin{prop}\label{surjfyt_odd}
The map $\Sha_{T}$ surjects onto $N^{Y,T}$.
\qed
\end{prop}

Consider the action of $N_{H}(T)$ on $M \times T_{\rm reg}$ given by:
\begin{equation*}
w : ((m, \gm)) \mapsto (m \Delta(w) , w^{-1} \gm w),
\end{equation*}
for $w \in N_{H}(T)$.  

\begin{prop} 
Let $m \in M$ and $\gm \in T_{\rm reg}$.  The fibre of $\Sha_{T}$ containing $(m , \gm)$ is equal to
\begin{equation*}  
\{( m \Delta( w) ,  w^{-1} \gm  w )\mid   w \in N_{H}(T) \}.
\end{equation*}
\qed
\end{prop}
  
\begin{cor} 
The map $\Sha_{T}$ descends to a surjective map  
\begin{equation} \label{real_Fy_odd}
\Sha_{T}: M/\Delta_T \times T_{\rm reg}\to N^{Y,T}.
\end{equation}
The fibres of $\Sha_{T}$ are the same as the $W_{H_{Y}}(T)$-orbits on $M/\Delta_{T} \times T_{\rm reg}$.  In other words, the fibre of $\Sha_{T}$ containing $(m,\gm)$ is equal to
\begin{equation*}
\{ (m \Delta( w) , w^{-1} \gm  w )\mid  w \in W_{H_{Y}}(T) \}.
\end{equation*}
\qed
\end{cor}

The maps defined by (\ref{real_Fy}) and (\ref{real_Fy_odd}) are the ones we will refer to by $\Sha_{T}$ henceforth.

\section{The split classical case and quasisplit unitary case} \label{MatrixSection}

In this section we write everything out explicitly for a representative set of examples.    
 
 \subsection{Set-up}\label{MatrixSection_1}
Let $D,k,m$ be nonnegative integers, with $D=m+2k$ and $m \geq k$.  We will put $V=F^D$ in the orthogonal/symplectic cases, and $V=E^D$ in the unitary case.

In this section $e_i \in F^D$ denotes the standard $i$th basis vector as usual.
Let $W=\Span \{ e_1, \ldots, e_k\}$ and let $W'=\Span \{ e_{D+1-k}, \ldots, e_{D}\}$. Let $X=\Span \{ e_{k+1}, \ldots, e_{D-k}\}$, $Y=\Span \{ e_{k+1}, \ldots, e_{2k} \}$, and $Y^\perp=\Span \{ e_{2k+1}, \ldots, e_{D-k} \}$.
(Here `$\Span$' refers to the $F$-span in the orthogonal/symplectic cases, and to the $E$-span in the unitary case.) The reader can easily check that in each of the following cases below the subspace $Y$ will indeed be nondegenerate for the form $\Phi$ that we are going to specify, and moreover, $Y^{\perp}$ will be its orthogonal complement in $X$. Thus the maps $\xi_{Y}$, $\theta$ are as defined in the previous sections. We will now make these maps explicit.  
   
Define $\xi: X \to W$ via the matrix $\xi=\left(\begin{array}{ccc}
I_{k}&0
\end{array}\right)$, and $\xi^*: W' \to X$ via the matrix $\xi^{*}=\left(\begin{array}{c}
I_{k}\\
0
\end{array}\right)$.  Thus the isomorphism $\upsilon: W \to W'$ is represented by the identity matrix $I_k$.

For an integer $s \geq 1$, write $J_+(s)$ for the $s \times s$ symmetric matrix 
\begin{equation*}
J_{+}(s)= \left(\begin{array}{ccccc}
&&&&1\\
&&&1&  \\
&&\Ddots&&\\
&1&&&\\
1&&&&
\end{array}\right),
\end{equation*}

and, when $s$ is even, write $J_-(s)$ for the $s \times s$ matrix
\begin{equation*}
J_-(s)= \left(\begin{array}{ccccc}
&&&&1\\
&&&-1&  \\
&&\Ddots&&\\
&1&&&\\
-1&&&&
\end{array}\right).
\end{equation*}
 
Let $T_G \leq \GL(W)$ be the subgroup of diagonal matrices, and write $\chi_i \in X^*(T_G), 1 \leq i \leq k$ for the character of $T_G$ given by taking the $i$th diagonal entry.

\subsection{Symmetric case} \label{ping}
 
Put

\begin{equation*}
J_V=\left(\begin{array}{cccc}
&&&J_{+}(k)\\
&J_{+}(k)&&  \\
&&J_{+}(m-k)&\\
J_{+}(k)&&&
\end{array}\right).
\end{equation*}

Let $\Phi$ denote the bilinear form defined on $V$ by $J_V$, so that 
\begin{equation*}
\Phi(v,v')=\ ^{t} vJ_V v',\ \ v,v'\in V.
\end{equation*}
 
The automorphism $\theta: G \to G$ is given by $\theta(g)=J^{-1} g^{-t} J$, with $J=J_{+}(k)$.
 Write $T$ for the torus in $H_{Y}$ consisting of matrices of the form
\begin{equation*}
\gamma=\left(\begin{array}{cccccc}
t_{1}&&&&&\\
&t_{2}&&&& \\
&&\ddots&&&  \\
&&&t_{2}^{-1}&&\\ 
&&&&t_{1}^{-1}& \\
&&&&&I_{m-k}
\end{array}\right)_X,
\end{equation*}
 relative to the above basis of $X$.  Let us first treat the case of $k$ even. If $\gamma-1$ is invertible on $Y$, then we have
\begin{equation*}
\gamma_{G}=\left(\begin{array}{cccccc}
t_{1}-1&&&&\\
&t_{2}-1&&& \\
&&\ddots&& \\
&&&t_{2}^{-1}-1&\\ 
&&&&t_{1}^{-1}-1
\end{array}\right)^{-1}_W,
\end{equation*}
relative to the above basis of $W$.

If $k$ is odd, then we must form $n_Y(\epsilon_Y \gm)$.  Here $\epsilon_Y=\left(\begin{array}{cc}
-I_k & 0 \\
0 & I_{m-k} \\
\end{array} \right)_X$, and taking $\gm$ so that $\gm+1$ is invertible on $Y$, we similarly have
\begin{equation*}
(\epsilon_Y \gamma)_{G}=-\left(\begin{array}{cccccc}
t_{1}+1&&&&\\
& \ddots&&& \\
&&2&& \\
&&& \ddots &\\ 
&&&&t_{1}^{-1}+1
\end{array}\right)^{-1}_W.
\end{equation*}

In both cases ($k$ even/odd), the centralizer $Z_G({}^{\xi}T)=T_G$.  

\subsection{Antisymmetric case} \label{pong}

We only consider the antisymmetric case when $m$ and $k$ are even.  Put

\begin{equation*}
J_V=\left(\begin{array}{cccc}
&&&J_{-}(k)\\
&J_{-}(k)&&  \\
&&J_{-}(m-k)&\\
J_{-}(k)&&&
\end{array}\right).
\end{equation*}

Let $\Phi$ denote the bilinear form defined on $V$ by $J_V$:
\begin{equation*}
\Phi(v,v')=\ ^{t} vJ_V v',\ \ v,v'\in V.
\end{equation*}
 
The automorphism $\theta: G \to G$ is given by $\theta(g)=J^{-1} g^{-t} J$, with $J=J_-(k)$.  Write $T$ for the torus in $H_{Y}$ consisting of matrices of the form
\begin{equation*}
\gamma=\left(\begin{array}{cccccc}
t_{1}&&&&&\\
&t_{2}&&&& \\
&&\ddots&&&  \\
&&&t_{2}^{-1}&&\\ 
&&&&t_{1}^{-1}& \\
&&&&&I_{m-k}
\end{array}\right)_X,
\end{equation*}
 relative to the above basis of $X$.  When $\gamma-1$ is invertible on $Y$,
we have
\begin{equation*}
\gamma_{G}=\left(\begin{array}{cccccc}
(t_{1}-1)^{-1}&&&&&\\
&\ddots&&&&  \\
&&(t_{\frac{k}{2}}-1)^{-1}&&&\\ 
&&&(t_{\frac{k}{2}}^{-1}-1)^{-1}&&\\
&&&&\ddots&\\
&&&&&(t_{1}^{-1}-1)^{-1}
\end{array}\right)_W.
\end{equation*}

Again in this case, the centralizer $Z_G({}^{\xi}T)=T_G$.   

\subsection{Hermitian case} \label{Hermitian.Case}
Now we set up a ``standard'' quasisplit Hermitian case.  Let $r=\left[ \frac{k}{2} \right]$.  Put

\begin{equation*}
J_V=\left(\begin{array}{cccc}
&&&J_{+}(k)\\
&J_{+}(k)&&  \\
&&J_{+}(m-k)&\\
J_{+}(k)&&&
\end{array}\right),
\end{equation*}

and let $\Phi$ denote the sesquilinear form defined on $V$ by $J_V$: 
\begin{equation*}
\Phi(v,v')=\ ^{t}\ol vJ_V v',\ \ v,v'\in V.
\end{equation*}
 
The automorphism $\theta: G \to G$ is given by $\theta(g)=J^{-1} (\ol g)^{-t} J$, with $J=J_+(k)$.

 Next, consider the maximal torus $T$ in $H_{Y}$ consisting of matrices of the form
\begin{equation*}
\gamma=\left(\begin{array}{cccccc}
t_{1}&&&&&\\
&t_{2}&&&& \\
&&\ddots&&&  \\
&&&\ol {t_{2}}^{-1}&&\\ 
&&&&\ol {t_{1}}^{-1}& \\
&&&&&I_{m-k}
\end{array}\right)_X,
\end{equation*}
with $t_i \in E^{\times}$, relative to the above basis of $X$.  Note that if $k$ is odd, then $t_{r+1} \in E^{\times}$ has norm $1$.

Write $S \leq T$ for the maximal $F$-split subtorus of $T$, consisting of $\gm$ as above with $t_i \in F^{\times}$.

Thus ${}^\xi S$ is given by 
\begin{equation*}
 \left(\begin{array}{cccccc}
t_{1}&&&&&\\
&\ddots&&&&  \\
&&t_{r}&&&\\ 
&&& t_{r}^{-1}&&\\
&&&&\ddots&\\
&&&&&t_{1}^{-1}
\end{array}\right)_W,
\end{equation*}
with $t_i \in F^{\times}$, when $k$ is even.  When $k$ is odd, the element $t_{r+1}$ appears as the middle entry.

For $\gm \in T$ with $\gamma-1$ invertible on $Y$, we have
\begin{equation*}
\gamma_{G}=\left(\begin{array}{cccccc}
(t_{1}-1)^{-1}&&&&&\\
&\ddots&&&&  \\
&&(t_{r}-1)^{-1}&&&\\ 
&&&(\ol t_{r}^{-1}-1)^{-1}&&\\
&&&&\ddots&\\
&&&&&(\ol t_{1}^{-1}-1)^{-1}
\end{array}\right)_W \in GL_E(W)
\end{equation*}

as above when $k$ is even.  When $k$ is odd, the middle entry is of course $(t_{r+1}-1)^{-1}=(\ol t_{r+1}^{-1}-1)^{-1}$.

We have $T_G=Z_G({}^{\xi}S)=Z_G({}^{\xi}T)$.  Write $S_G$ for the maximal $F$-split torus in $T_G$; it is given by the diagonal matrices in $G$ with entries in $F^{\times}$.    
 
 \subsection{The section}
Finally, in the symplectic case, the unitary case, and the orthogonal case with $\dim W$ even, our matrix $n_{Y}(\gamma) \in N$ is written by fitting together the above matrices via

\begin{equation*}
n_{Y}(\gm)=\left(\begin{array}{ccc}
I_k& \xi &\gm_G \\
&I_m&-\xi^*  \\
&&I_k
\end{array}\right)_V.
\end{equation*}
(Our choice of basis forces the factor $\upsilon^{-1}$ to be the identity.) In the case when $V$ is orthogonal and $\dim W$ odd, the matrix $n_{Y}(\epsilon\gamma)$ is the same except that the entry $\gm_{G}$ is replaced by $(\epsilon\gm)_{G}$.

\section{Density in $N$}\label{densarg}

In this section we show that the union of the sets $N^{Y,T}$, as $Y$ runs over $\mathcal Y_{k}$ and $T$ varies over conjugacy classes of maximal tori in $H_{Y}$, is open and dense in $N$. In what follows, unless mentioned otherwise, the topology in question is the Zariski topology. The primary goal is to establish a subset of $H$ which is nonempty and open in the image of $\Norm$. We first deal with the case when the base field $F$ is algebraically closed. Let boldface notation indicate the $\ol F$-points of various varieties.  Put $k=\dim_{F} W$ and $m=\dim_{F} X$ if $H$ is a symplectic or an orthogonal group. Put $k=\dim_{E} W$ and $m=\dim_{E} X$ if $H$ is a unitary group in $m$ variables. 

Recall our notational convention, according to which in the unitary case $\bT_{\rm reg}$ denotes the set of regular elements $\gm \in \bT$ so that $(\gm -1)|_Y$ is invertible.

\begin{defn}  
For a subvariety $\bf A \subseteq \bH$, put
\begin{equation*}
\bA_k=\{ h\in \bA \mid \rank(h-1;\bX) \leq k \}.
\end{equation*}
%\begin{equation*}
%\bA_{=k}=\{ h\in \bA \mid \rank(h-1;\bX) = k \},
%\end{equation*}
 \end{defn}
Clearly, $\bA_k$ is closed in $\bA$.  Note that the image of the $\Norm$ map lies in $\bH_{k}$.  Denote by $\bA_{k,{\rm gen}}$ the subset of $\bA_k$ consisting of elements which have the eigenvalue 1 occuring with multiplicity $m-k$ and all other eigenvalues occurring with multiplicity one. For $h\in \bA_{k,{\rm gen}}$, $h-1$ has as many distinct eigenvalues as its rank, so that, we have:
\begin{lemma}\label{hgenss}
If $h\in \bA_{k,{\rm gen}}$, then $h$ is semisimple.
\qed
\end{lemma} 

Let $\bY\subseteq \bX$ be a nondegenerate subspace of dimension $k$. As observed in Proposition \ref{scrsk}, in all the cases treated in this article, there is precisely one nondegenerate subspace of $\bX$ of a fixed dimension up to translation by $\bH$. Let $\bT$ be a maximal torus of $\bH_{\bY}$.

As in Section \ref{algebraic_th}, we deal separately with the cases when $\bH_{\bY}$ is odd orthogonal or otherwise.
\subsection{$\bH_{\bY}$ is not odd orthogonal}\label{density_noo}
\begin{defn}
Denote by $\bH^{\bT}$ the set $\Int(\bH)(\bT_{\rm reg})$.
\end{defn}
 
\begin{lemma}\label{conj_gen}
\begin{enumerate}
\item We have $\bT_{\rm reg}=\bT_{k,{\rm gen}}=\bH_{k,{\rm gen}} \cap \bT$. 
\item $\bH^{\bT}=(\bH^{\circ})_{k,{\rm gen}}.$
\end{enumerate}
\end{lemma} 
\begin{proof}
The first statement is straightforward. The second statement is clear when $H$ is a unitary group so assume it is not so. Note that $\Int(\bH)(\bT_{\rm reg})\subseteq (\bH^{\circ})_{k,{\rm gen}}$. Let $h\in (\bH^{\circ})_{k,{\rm gen}}$ and define $\bY'=(\bX^{h})^{\perp}$. By Lemma \ref{iamalemma} and Lemma \ref{hgenss}, $\bY'$ is a nondegenerate space of dimension $k$.  By Proposition \ref{scrsk} (ii), there exists $h_{1}\in \bH$ be such that $h_{1}(\bY')=\bY$. Then $h_{1}\bH_{\bY'}h_{1}^{-1}=\bH_{\bY}$ and so, $h_{1}hh_{1}^{-1}$ is a semisimple element of $\bH_{\bY}^{\circ}$. Since $\bT$ is the unique maximal torus in $\bH_{\bY}$ up to conjugacy, $h_{1}hh_{1}^{-1}$ can be conjugated in $\bH_{\bY}$ to an element $t$ of $\bT$. Clearly $t\in\bT_{\rm reg}$, and so $h \in {\bH}^{\bT}$.
\end{proof}
 
\begin{prop} \label{Zariski} 
$(\bH^{\circ})_{k,{\rm gen}}$ is open in $\bH_{k}$.
 \end{prop}
 \begin{proof}
Let $p_{h}(x)$ be the characteristic polynomial $\det(h-xI)$ of an $h\in \bH$. If $h\in\bH_{k}$, then $(1-x)^{m-k}|p_{h}(x)$. Write $p_{h}(x)=(1-x)^{m-k}q_{h}(x)$. Note that $h\in (\bH^{\circ})_{k,{\rm gen}}$ if and only if $q_{h}(x)$ has distinct roots and none of them is 1. For an $h\in \bH_{k}$, let $q_{h}(x)=\Sigma_{i=0}^{k-1}c_i(h)x^{i}+(-1)^{k}x^{k}$. Define $\chi:\bH_{k}\to \mathbb A^{k}$ such that $\chi(h)=(c_0(h),\dots,c_{k-1}(h))$.  On the other hand, given $v=(c_{0},\dots,c_{k-1})\in\mathbb A^{k}$, denote by $q_{v}$ the polynomial $\Sigma_{i=0}^{k-1}c_{i}x^{i}+(-1)^{k}x^{k}$. Define
 $${\bf U}=\{v\in \mathbb A^{k}\mid  \ q_{v}(1) \neq 0, \disc(q_v) \neq 0  \},$$
 where $\disc$ denotes the usual discriminant of a polynomial.
 
 Note that ${\bf U}$ is open in $\mathbb A^{k}$. 
 Now $\chi(h) \in {\bf U}$ if and only if the algebraic multiplicity of $1$ as an eigenvalue of $h$ is $m-k$, and the other eigenvalues of $h$ are distinct.
 Thus $(\bH^\circ)_{k,{\rm gen}}=\chi^{-1}({\bf U}) \cap \bH^\circ$, and hence the statement.
\end{proof}

\begin{defn}
Let $\bN_{{\bf {\rm reg}}}=\{ n \in \bN' \mid \Norm(n) \in \bH^{\bT} \}$, and let $N_{{\rm reg}}$ be the $F$-points of $\bN_{{\bf {\rm reg}}}$.
\end{defn}

(The proof of the previous proposition gives an $F$-structure on  $(\bH^{\circ})_{k,{\rm gen}}=\bH^{\bT}$, and thus   an $F$-structure on  $\bN_{{\bf {\rm reg}}}$.)

\begin{thm} \label{Theorem 1}
\begin{enumerate}
\item $\bN_{{\bf {\rm reg}}}$ is nonempty and open in $\bN$.
\item  $N_{{\rm reg}}$ is the union of $N^{Y,T}$, for $Y $ runs over $ \mc Y_k$  (see Proposition \ref{scrsk}) and as $T$ runs over conjugacy classes of maximal tori in $H_{Y}$.
\item We have a decomposition
\begin{equation} \label{N_decomp}
N_{{\rm reg}}= \bigsqcup_{Y,T}\{ \Int(m) n_{Y}(\gm) \mid m \in M, \gm \in T_{\rm reg} \}.
\end{equation}
\end{enumerate}
\end{thm}

\begin{proof} 
Recall that the image of $\Norm$ lies in $\bH_{k}$.
Since $\bH^{\bT}$ is open in $\bH_{k}$, we see that $\bN_{{\bf {\rm reg}}}$ is open in $\bN'$, which in turn is open in $\bN$.  One can apply $n_{\bY}$ to elements of $ \bT_{\rm reg}$ to produce elements of $\bN_{{\bf {\rm reg}}}$, so it is nonempty. This proves the first statement. If $n\in N^{Y,T}$ for some $Y$ and $T$, then it is clear that $n\in N_{{\rm reg}}$. So to prove the second statement assume $n\in N_{\rm reg}$. By definition of $N_{{\rm reg}}$, it is clear that $\Norm(n)$ is semisimple. Let $X^{\Norm(n)}=Y^{\perp}$ which is nondegenerate by Lemma \ref{iamalemma} and $(m-k)$-dimensional. Thus $\Norm(n)\in T_{\rm reg}$ for some maximal torus $T \leq H_{Y}$, which gives us the desired statement. The last statement follows from the second part and Proposition \ref{surjprop}.
\end{proof}

\subsection{$\bH_{\bY}$ is odd orthogonal}\label{density_oo}
Let $\bH^{+}=\bH-\bH^{\circ}$. Recall that $\epsilon_{\bY}=P_{\bY^{\perp}}-   P_{\bY}  \in\bH_{\bY}$. 
\begin{defn}
Put $\bH^{\bT}=\Int(\bH)(\epsilon_{\bY}\bT_{\rm reg}).$ \end{defn}

\begin{lemma}\label{conj_gen_odd}
$\bH^{\bT}=(\bH^{+})_{k,{\rm gen}}.$
\end{lemma} 
\begin{proof}
As in the proof of Lemma \ref{conj_gen}, given $h\in \bH^{+}_{k,{\rm gen}}$, there exists $h_{1}\in \bH$ such that $h_{1}hh_{1}^{-1}$ is a semisimple element of $(\bH^+)_{\bY}$. Thus $\epsilon_{\bY}h_{1}hh_{1}^{-1}$ can be conjugated to an element $t\in \bT$. Since $\epsilon_{\bY}$ commutes with elements of $\bH_{\bY}$, $h$ can be conjugated in $\bH$ to $\epsilon_{\bY}t$. It is easy to see that $t\in \bT_{\rm reg}$. The containment $\bH^{\bT}\subseteq (\bH^{+})_{k,{\rm gen}}$ is obvious.
\end{proof}

The proofs of the next two statements are similar to the proofs of the corresponding statements in Section \ref{density_noo}.
\begin{prop} \label{Zariski_odd} 
$(\bH^{+})_{k,{\rm gen}}$ is open in $\bH_{k}$.
\qed
 \end{prop}

\begin{defn}
As in the earlier case, define $\bN_{{\bf {\rm reg}}}$ to be the preimage  in $\bN'$ of the $\Norm$ map of $\bH^{\bT}$, and $N_{{\rm reg}}$ to be the $F$-points of $\bN_{{\bf {\rm reg}}}$.
\end{defn}

\begin{thm} \label{Theorem 1_odd}
\begin{enumerate}
\item $\bN_{{\bf {\rm reg}}}$ is nonempty and open in $\bN$.
\item $N_{{\rm reg}}$ is the union of $N^{Y,T}$, for $Y $ runs over $ \mc Y_k$ and as $T$ runs over conjugacy classes of maximal tori in $H_{Y}$.
\item We have a decomposition
\begin{equation} \label{N_decomp_odd}
N_{{\rm reg}}= \bigsqcup_{Y,T} \{ \Int(m) n_{Y}(\epsilon_{Y}\gm) \mid m \in M, \gm \in T_{\rm reg} \}.
\end{equation}
\end{enumerate}
\qed
\end{thm}

\section{Lie algebra decompositions} \label{Lie}

An element $u\in \mathfrak n$ is characterized by linear maps
$$C:X\to W, \ C':W'\to X, \ D:W'\to W$$
such that $u|_{W}=0$, $u|_{X}=C$, and $u|_{W'}=C'+D$.
The condition that $u\in \mathfrak n$ is equivalent to the two conditions:
\begin{enumerate}
\item $C^{*}+C'=0$,
\item $D^{*}+D=0$. 
\end{enumerate}
So we write $u\in \mathfrak n$ as $u(C,D)$ where the condition is simply that $D$ is skew-Hermitian on $W'$.

\subsection{Four exact sequences} \label{4X}

The network of relationships between the Lie algebras of the groups that we have encountered so far is subtle, governed by no fewer than four exact sequences of vector spaces.  
Given $Y \in \mc Y_k$, fix as before a surjection $\xi: X \to W$ so that $\ker \xi=Y^\perp$.  Recall that $\xi^*: W' \to X$ is injective with image $Y$, and moreover $\xi^*(\xi \xi^*)^{-1}\xi$ is the projection $P_Y$.

Our starting point is the map
 $\phi:\mm=\mathfrak g \oplus \mathfrak h \to \mathfrak n$ given by
 \begin{equation*}
 \phi(A,B)=u(A\xi,\xi B\xi^{*}).
 \end{equation*}
This map is typically neither injective nor surjective.  Let us identify its kernel and image.
 
\begin{defn}
Let 
\begin{equation*}
\kappa= \{ B \in \hh \mid B(Y) \subseteq Y^{\perp} \}.
\end{equation*}
Define $i_{2}:\kappa\to \mathfrak m$ by $i_{2}(B)=(0,B)$.
Also we set
$$\mathfrak n_{\xi}=\{u(C,D) \in \nn \mid  Y^{\perp}\subseteq \ker C\}.$$
\end{defn}

\begin{prop}  (The Exact Sequence for $\mm$) The following sequence is exact:
\begin{equation} \label{M_seq}
0 \to \kappa\overset{i_{2}}\to \mathfrak m \overset{\phi}\to \mathfrak n_{\xi} \to 0.
\end{equation}
\qed
\end{prop}
  
We next identify the quotient of $\nn$ by $\nn_{\xi}$.

\begin{defn}
 Let $i$ denote the inclusion of $\mathfrak n_{\xi}$ into $\mathfrak n$. 
Write $\psi: \mathfrak n \to \Hom(W',Y^{\perp})$ for the map given by
\begin{equation*}
\psi: u(C,D)\mapsto P_{Y^{\perp}}\circ C^{*}.
\end{equation*}
\end{defn}

\begin{prop} \label{N_exact} (The Exact Sequence for $\nn$) The following sequence is exact:
\begin{equation} \label{N_seq}
0\to \mathfrak n_{\xi}\overset{i}\to \mathfrak n \overset{\psi}\to \Hom(W',Y^{\perp}) \to 0.
\end{equation}
\qed
\end{prop}
Define $\tilde \psi: \Hom(W',Y^{\perp}) \to \nn$ via $\tilde \psi(A) =u(A^*P_{Y^{\perp}},0)$; then $\tilde \psi$ is a section of $\psi$.  The space $\kappa$ also maps onto $\Hom(W',Y^{\perp})$ in the following manner.

 \begin{defn}  Define $\Gamma_{1}:\kappa\to \Hom(W',Y^{\perp})$ via $\Gamma_{1}(B)=B\xi^{*}$. Further define $\Gamma_{2}:\Hom(W',Y^{\perp})\to \hh$ via $\Gamma_{2}(A)=A(\xi\xi^{*})^{-1}\xi- \xi^{*}(\xi\xi^{*})^{-1}A^{*}=A((\xi|_{Y})^{*})^{-1}P_{Y}- (\xi|_{Y})^{-1}A^{*}$. 
\end{defn}

One checks that the image of $\Gamma_2$ lands in $\kappa$, and that $\Gamma_1 \circ \Gamma_2$ is the identity on $\Hom(W',Y^{\perp})$.

\begin{prop} \label{kappaX} (The Exact Sequence for $\kappa$) 
The following sequence is exact:
\begin{equation} \label{kappa_seq}
0 \to \hh_{Y^{\perp}} \overset{i}\to \kappa \overset{\Gamma_{1}}\to \Hom(W',Y^{\perp}) \to 0.
\end{equation}
Here $i$ is the natural inclusion.
\qed
\end{prop}
 
By Lemma \ref{basic xi prop}, we have 
\begin{equation} \label{blabel}
z_{\hh}(\mathfrak t)=\hh_{Y^{\perp}} + \mathfrak t.
\end{equation}
The sum is direct so $\hh_{Y^{\perp}}$ is the quotient of $z_{\mathfrak h}(\mathfrak t)$ by $\mathfrak t$. We make this explicit as follows:

\begin{defn} 
 Define the map $K:z_{\mathfrak h}(\mathfrak t) \to \hh_{Y^{\perp}}$ by $K(A)=AP_{Y^{\perp}}$.
\end{defn}

\begin{prop} (The Exact Sequence for $z_{\mathfrak h}(\mathfrak t)$) The following sequence is exact:
\begin{equation} \label{Z_seq}
0\to \mathfrak t\overset{i}\to z_{\mathfrak h}(\mathfrak t) \overset{K}\to \hh_{Y^{\perp}} \to 0.
\end{equation}
Here $i$ is the natural inclusion.
\qed
\end{prop}
 
We will later use these exact sequences to calibrate our differential forms on these spaces.  For the moment, we use them to make a simple observation about dimensions.

\begin{cor} \label{dimension_formula}
\begin{equation*}
\dim M-\dim Z_H(T)+ \dim T=\dim N.
\end{equation*}
\end{cor}
\begin{proof}
Follows directly from taking the alternating sum of equalities of dimensions obtained from the exact sequences in (\ref{M_seq}), (\ref{N_seq}) (\ref{kappa_seq}) and (\ref{Z_seq}).
\end{proof}

 \subsection{Decompositions} \label{Picture_Pages}

\begin{prop} We have a direct sum decomposition
\begin{equation*}
\hh=\hh_{Y^{\perp}} + \hh_{Y}+ \im \Gamma_2.
\end{equation*}
\end{prop}

\begin{proof} Since $\Gamma_2$ splits the sequence (\ref{kappa_seq}), we have a direct sum $\kappa=\hh_{Y^\perp}+ \im \Gamma_2$.  It is easy to see that $\hh=\kappa+ \hh_Y$ is direct.
\end{proof}
 
Much of this can be visualized with matrices; we offer the following representative image in the case when $H$ is symplectic:

\colorlet{lightgray}{black!15}
\begin{center}
\begin{tikzpicture}
 \fill[color=lightgray] (-1,3) rectangle (1,-3);
 \fill[color=lightgray] (-3,1) rectangle (3,-1);
 \fill[color=gray] (-1,1) rectangle (1,-1);
 
 \draw (-1,3) -- (-1,-3);
 \draw (1,3) -- (1,-3);
 \draw (-3,-1) -- (3,-1);
 \draw (-3,1) -- (3,1);
 \draw[line width=0.05cm] (-3,3) arc (165:195:12);
 \draw[line width=0.05cm] (3,-3) arc (-15:15:12);
 %\draw[style=dashed] (-3,-3) -- (3,3);
 \draw[style=dashed] (-3,3) -- (3,-3);
 \path (-2,3.5)node (A) {$Y$};
 \path (0,3.5)node (B) {$Y^{\perp}$};
 \path (2,3.5)node (C) {$Y$};
 \path (-4,-2)node (D) {$Y$};
 \path (-4,0)node (E) {$Y^{\perp}$};
 \path (-4,2)node (F) {$Y$};
 \draw (A); \draw (B); \draw (C); \draw (D); \draw (E); \draw (F);
\end{tikzpicture}
\captionof{figure}{Composition of $\hh$}
 \label{Fig_1}

\end{center}

Figure \ref{Fig_1} represents $\hh$ ($H$ is a symplectic group), viewed as matrices in $\mathfrak{gl}(X)$.  The middle rows and columns corresponding to $Y^{\perp}$ are shaded (both light gray and dark gray); this is $\kappa$.  The subalgebra $\hh_{Y^\perp}$ corresponds to the dark gray central block. The Lie algebra $z_{\HH}(\mathfrak t)$ is the sum of the dark gray block and the diagonal. The light gray region corresponds to the image of $\Hom(W',Y^{\perp})$ under $\Gamma_{2}$.  
The unshaded region is $\hh_{Y}$.
 The map $\Xi$ can be visualized by simply deleting the five shaded regions.  
 
% \bigskip
 
Meanwhile, $\nn$ can be expressed as a direct sum $\nn=\nn_\xi + \im \tilde \psi$, by Proposition  \ref{N_exact}.
 The following is a diagram of an essential piece of $\nn$:

\colorlet{lightgray}{black!15}
\begin{center}
\begin{tikzpicture}  
 \fill[color=lightgray] (-1,3) rectangle (1,1);
  
 \draw (-1,3) -- (-1,1);
 \draw (1,3) -- (1,1);
 
 \draw (3,3) --(3,1);
  
 \draw[line width=0.05cm] (-3,3) arc (165:195:4);
 
    \draw[line width=0.05cm] (5,1) arc (-15:15:4);

 \path (-2,3.5)node (A) {$Y$};
 \path (0,3.5)node (B) {$Y^{\perp}$};
 \path (2,3.5)node (C) {$Y$};
 \path (4,3.5) node (D) {$W'$};
 \path (4,2)node (D) {$D$};
 %\path (-4,0)node (E) {$W$};
 \path (-4,2)node (F) {$W$};
 \draw (A); \draw (B); \draw (C); \draw (D); \draw (E); \draw (F);
\end{tikzpicture}

\captionof{figure}{Blow-up of $\nn$}
 \label{Fig_2}
 \end{center}

Figure \ref{Fig_2} is a blow-up of the $(C,D)$ portion of
$\left(\begin{array}{ccc}
0& C &D\\
&0&  -C^*\\
&&0
\end{array}\right) \in \nn.
$
The shaded region corresponds to $\Hom(W',Y^{\perp}) \cong \Hom(Y^{\perp},W)$, or the image of $\tilde \psi$.  The unshaded area corresponds to $\nn_\xi$. The skew-symmetric matrices $D$ indicate the image of $ 0 \oplus \hh_{Y}$ under $\phi$.  The rest of the unshaded area corresponds to the image of $\GG \oplus 0 \subseteq \mm$ under $\phi$.

\section{Jacobian of $\Sha_T$: first steps} \label{local_now}
Henceforth in this paper we furthur suppose that $F$ is a local field. 

Our computation of the Jacobian in this paper is modeled on that of \cite{IFS}, which itself is modeled on the proof of the Weyl Integration Formula in \cite{BourbLie}.

\subsection{Tangent spaces}\label{ss_ts}
In this section we explain certain identifications of tangent spaces for later use.  Generally if $g \in G$, one has left translation $\lambda_g$ and right translation $\rho_g$.  The differentials of these maps at $1 \in G$ are surjective, i.e.,
%(see Bourbaki III, Section 2.2),
\beq
T_gG=(d \lambda_g)_1(\mf g)=(d \rho_g)_1(\mf g).
\eeq
 Given $g \in G$ and $X \in \mf g$, we set $gX=(d \lambda_g)_1(X)$ and $Xg=(d \rho_g)_1(X)$.
 
Let $U= \Hom(X,W) \times \Hom(W',W)$, and write $i(n(\xi,\eta))=(\xi,\eta)$.  Then $i: N \to U$ is a submanifold, with $i(1)=(0,0)$.  Its differential $di_1$ is an isomorphism onto $\mf n$.
For $n \in N$, there are obvious extensions of $\lambda_n$ and $\rho_n$ to $U$; we give the extensions the same names.  Namely, if $n=n(\xi,\eta)$, then 
 \beq
 \rho_n(u(C,D))=n(C+\xi,D+\eta-C \xi^*),
  \eeq
  and similarly for $\lambda_n$.

Through the commutative diagram
\begin{equation}  
	\xymatrixcolsep{5pc}\xymatrix{
		  T_1(N) \ar[r]^{ di_1} \ar[d]_{(d\rho_n)_1} & U \ar[d]^{(d\rho_n)_{(0,0)}} \\
		T_n(N)  \ar[r]^-{ di_n}& U \\
		},
\end{equation}
we may identify $T_n(N)$ with the image of $di_n$, which is equal to  $(d\rho_n)_0(\mf n)$.

Similarly for $\gm \in T$ we identify $T_\gm(T)$ with $(d \rho_\gm)_0(\mf t) \subseteq \Hom(X,X)$, with a typical element written simply as $\gm Z$ when $Z \in \mf t\subset \Hom(X,X)$.

\subsection{Derivative of $\Sha_{T}$}

As earlier we bifurcate into the cases when $H_{Y}$ is odd orthogonal or otherwise.    Recall the map 
\begin{equation*}
\Sha=\Sha_{T}: M/\Delta_{T} \times T_{\reg}\to N.
\end{equation*}
Let us first treat the case when $H_{Y}$ is not odd orthogonal. The derivative
\begin{equation*}
d\Sha_{(1 , \gamma)} : T_1( M/\Delta_{T}) \oplus T_{\gm}( T) \to T_{n(\gm)}N
\end{equation*}
 is straightforward to compute:
\begin{prop} 
For $A\in \mathfrak g$, $B\in \mathfrak h$ and $Z\in \mathfrak t$, 
we have
\begin{equation} \label{the_expression}
d\Sha_{(1,  \gamma)}((A , B),\gamma Z)=u(A\xi-\xi B, \gm_G \ups^{-1}A^{*}+A\gm_G \ups^{-1}- \gm_G^2 \Xi(\gm Z) \ups^{-1}).
\end{equation}
\qed
\end{prop}
 
The derivative at $(m , \gamma) \in  M/\Delta_{T} \times T_{\reg}$ can be inferred from (\ref{the_expression}) through the commutative diagram:

 \begin{equation} \label{trans_diagram} 
	\xymatrixcolsep{5pc}\xymatrix{
		  T_1( M/\Delta_{T}) \oplus T_{\gm} T   \ar[r]^-{ d\Sha_{(1 , \gm)}} \ar[d]_{\lambda_m} & T_{\Sha(1 , \gm)}N\ar[d]^{\Ad(m)} \\
		T_m( M/\Delta_{T}) \oplus T_{\gm} T  \ar[r]^-{ d\Sha_{(m , \gm)}}& T_{\Sha(m, \gm)} N \\
		}
\end{equation}

Here $\lambda_{m}$ denotes left translation by $m \in M$.  For use afterwards, we multiply (\ref{the_expression}) on the right by $n(\gamma)^{-1}$ (as described in Section \ref{ss_ts}) to bring it to $T_{1}N=\nn$.  
This gives
$$\textrm{d}\Sha_{(1,\gamma)}((A, B), \gamma Z)n(\gamma)^{-1}=$$
\begin{align}\label{shadiff}
u(A\xi-\xi B, (A\xi-\xi B)\xi^{*}+\gm_G \ups^{-1}A^{*}+A\gm_G \ups^{-1}- \gm_G^2 \Xi(\gm Z) \ups^{-1}).
\end{align}

Now we consider the case when the group $H_{Y}$ is odd orthogonal. The expression for the derivative in this case can be obtained in the same way as above:
\begin{prop} 
For $A\in \mathfrak g$, $B\in \mathfrak h$ and $Z\in \mathfrak t$, 
we have
\begin{equation} \label{the_expression2}
d\Sha_{(1, \gamma)}((A,  B),\gamma Z)=u(A\xi-\xi B, (\epsilon\gamma)_G \ups^{-1} A^{*}+A(\epsilon\gamma)_G \ups^{-1}- (\epsilon\gm)_{G}^{2} \Xi(\epsilon\gm Z) \ups^{-1}).
\end{equation}
Also,
$$\textrm{d}\Sha_{(1,\gamma)}((A, B), \gamma Z)n(\epsilon\gamma)^{-1}=$$
\begin{align}\label{shadiff_2}
u(A\xi-\xi B, (A\xi-\xi B)\xi^{*}+(\epsilon\gamma)_G \ups^{-1}A^{*}+A(\epsilon\gamma)_G \ups^{-1}- (\epsilon\gm)_{G}^{2} \Xi(\epsilon\gm Z) \ups^{-1}).
\end{align}
\qed
\end{prop}

As before, the derivative at $(m, \gamma) \in  M/\Delta_{T} \times T_{\reg}$ can be inferred from (\ref{the_expression2}) through the diagram in (\ref{trans_diagram}).

\subsection{Top forms}
 Suppose that $\omega_N$ and $\omega_{T}$ are invariant top forms on $N$ and $T$, respectively, and that $\omega_{M/\Delta_{T}}$ is an $M$-invariant top form on $M/\Delta_{T}$.
(We will specify these later.)

\begin{defn} For $m \in M$, put
\begin{equation*}
\delta_N(m)=\det(\Ad(m); \nn).
\end{equation*}
\end{defn}

\begin{prop} \label{deltas}
There is a unique analytic function $\delta_T$ on $T_{\reg}$ so that
\begin{equation*}
\Sha_{T}^*(\omega_N)=\delta_N(m)\delta_T(\gm) \omega_{M/\Delta_T} \wedge \omega_T
\end{equation*}
at the point $(m,\gm) \in M/\Delta_T \times T_{\reg}$.
\end{prop}

\begin{proof}
Since $\Sha_{T}^*(\omega_N)$ and $\omega_{M/\Delta_T} \wedge \omega_T$ are both top forms on $M/\Delta_{T} \times T_{\reg}$, there is a unique analytic function
\begin{equation*}
\delta_T: M/\Delta_{T} \times T_{\reg} \to F
\end{equation*}
so that
\begin{equation*}
\Sha_{T}^*(\omega_N)=\delta_{T}(m,\gm) \omega_{M/\Delta_T} \wedge \omega_T.
\end{equation*}
Let $m_0 \in M$. Applying $\lambda_{m_0}^*$ to both sides of the equation gives
\begin{equation*}
\delta_N(m_0) \Sha_{T}^*(\omega_N) =\delta_{T}(m_0m,\gm)\omega_{M/\Delta_T} \wedge \omega_T.
\end{equation*}
We have used here (\ref{trans_diagram}) and the fact that $\Ad(m)^*\omega_N=\delta_N(m)\omega_N$.)
Therefore for all $m,m_0$, and $\gm$, we have
\begin{equation*}
\delta_{T}(m_0m,\gm)=\delta_N(m_0)\delta_T(m,\gm).
\end{equation*}
The result follows (put $\delta_T(\gm)=\delta_T(1,\gm)$).
\end{proof}

\subsection{Choice of bases}  \label{JacobianSection}
Let $\gm$ be a fixed element of $T_{\reg}$. To compute $\delta_{T}(\gm)$, we will essentially compute the top exterior power of $d\Sha_{(1,\gm)}$ with respect to fixed ordered bases $p\mathfrak B$ of $\mathfrak m/\Delta_{\mathfrak t} \oplus \mathfrak t$ and $\mathfrak B_{1}$ of $\nn$, that we will define below. These spaces have the same dimension by Corollary \ref{dimension_formula}. The basis $p\mathfrak B$ will depend on $\gm$, but in a manner dictated by translation in $T$. 

We now set up two bases of $\nn$, namely $\mathfrak B_1$ and $\mathfrak B_2(\gm)$, and a linear transformation $L(\gm)$ sending $\mathfrak B_1$ to $\mathfrak B_2 (\gm)$. %The basis $\mathfrak B_1$ will essentially be the input of $d \Sha_{(1,\gm)}$, and $\mathfrak B_2 (\gm)$ will essentially be the output. 
Then $\delta_{T}(\gm)$ will be equal to $\pm \det L(\gm)$, where the sign is independent of $\gm$. 

Consider a ordered basis $\mathfrak B$ of $\mathfrak m$ of the form
$$\{(\underline{A_{i}},0),(0,\Gamma_2(\underline{\beta_{j}})),(\Xi(\underline{B'_{j'}}),\underline{B'_{j'}}),(0,\underline{C_{k}}),(\Xi(\underline{Z_{l}}),\underline{Z_{l}})\}$$
 such that 
\begin{itemize}
\item $\{A_{i}\}$ is a basis of $\GG$,
\item $\{\beta_{j}\}$ is a basis of $\Hom(W',Y^{\perp})$,
\item  $\{B_{j'}'\}$ is a basis of $\hh_{Y^{\perp}}$,
\item $\{Z_{l}\}$ is a basis of $\mathfrak t$, and
\item $\{C_{k}\} \cup \{Z_{l}\}$ is a basis of $\hh_{Y}$.
\end{itemize}
(Recall that underlining the vectors of any of the above basis indicates a tuple; for instance $(\underline{A_{i}},0)$ indicates the ordered collection of the elements $(A_{1},0),\dots,(A_{x},0)$ where $x=\dim \GG$.) Note that $\{B_{j'}'\}\cup \{Z_{l}\}$ is a basis of $z_{\mathfrak h}(\mathfrak t)$. Moreover $\Xi(B'_{j'})=0$ although we won't require this fact here.  Put $B_j=\Gamma_2(\beta_j)$, so that $B_j\xi^{*}=\beta_{j}$.

We write $p\mathfrak B=p\mathfrak B(\gm)$ for
\begin{equation*}
p\mathfrak B= \left\{  (p(\underline{A_{i}},0);0);(p(0,\underline{B_j});0);(p(0,\underline{C_{k}});0);((0,0);\gm\underline{ Z_{l}}) \right\},
\end{equation*}
a basis of $\mathfrak m/\Delta_{\mathfrak t} \oplus \mathfrak t$.  Here $p$ is the projection from $\mm$ to $\mathfrak m/\Delta_{\mathfrak t}$.

Suppose first that $H_Y$ is not odd orthogonal. The following are two bases of $\mathfrak n$:
$$\mathfrak B_{1}=\{u(\underline{A_{i}}\xi,0),u(\underline{\beta_{j}^{*}},0),u(0,\xi\underline{C_{k}}\xi^{*}),u(0,\xi\underline{Z_{l}}\xi^{*})\},$$
\begin{equation*}
\mathfrak B_{2}(\gm)=\left\{u(\underline{A_{i}}\xi,-), u(\underline{\beta_{j}^{*}},0),u(-\xi\underline{C_{k}},-\xi\underline{C_{k}}\xi^{*}),u\left(0,- \gm_G^2 \Xi( \gm \underline{Z_{l}}) \ups^{-1}\right) \right\}
\end{equation*}
where
$$u(A_{i}\xi,-)=u\big( A_{i}\xi, A_{i}\ups^{-1}+ \gm_G \ups^{-1}A_{i}^{*} + A_{i}\gm_G \ups^{-1} \big).$$ 

(We extend $\beta_j^*$ from $Y^\perp$ to $X$ by setting it equal to $0$ on $Y$.)
To see that $\mathfrak B_1$ is a basis, note that $u(\underline{A_{i}}\xi,0)$, $u(0,\xi\underline{C_{k}}\xi^{*})$, and $u(0,\xi\underline{Z_{l}}\xi^{*})$ constitute a basis of $\nn_\xi$ and that
$u(\underline{\beta_{j}^{*}},0)$ gives a basis of $\im \tilde \psi$. It can be readily seen from this fact that $\mathfrak B_2(\gm)$ is a basis as well. 

Given $\gm \in T$, we define a linear transformation $L=L_{\mathfrak B}(\gm): \mf n \to \mf n$ so that:
\begin{itemize}
\item $L: u(A \xi,0) \mapsto u\big( A \xi, A\ups^{-1}+ \gm_G \ups^{-1} A^{*} + A\gm_G \ups^{-1} \big)$, for $A \in \mf g$,
\item $L: u(\beta^*,0) \mapsto u(\beta^*,0)$, for $\beta \in \Hom(W',Y^\perp)$,
\item $L: u(0,\xi C_{k}\xi^{*}) \mapsto u(-\xi C_{k},-\xi C_{k}\xi^{*})$,  
\item $L: u(0,\xi Z\xi^{*}) \mapsto u\left(0,- \gm_G^2 \Xi( \gm Z) \ups^{-1}\right)$, for $Z \in \mf t$.
\end{itemize}
It can be easily seen that the map $L$ is well defined. When $H_Y$ is odd orthogonal, we follow the same prescription except that we change $\gm$ to $\epsilon \gm$ throughout.
 
In the next section we will calibrate our forms $\omega_N$ and $\omega_{M/\Delta_{T} \times T}$ so that for all $\gm \in T_{\reg}$, we have
\begin{equation} \label{calibration}
\omega_{N}(1)[\mathfrak B_{1}]=\pm \omega_{M/\Delta_{T} \times T}((1 ,\gamma))[p \mathfrak B],
\end{equation}
for some uniform sign $\pm$ not depending on $\gm$.  Let us assume this now for the sake of exposition.

\begin{prop} \label{delta/det} 
  We have $\delta_T(\gm)=\pm  \det L (\gm) \neq 0$, with the sign being that of (\ref{calibration}).
 \end{prop}

\begin{proof}

We prove this when $H_Y$ is not odd orthogonal, the other case is similar.

Recall that $\delta_T$ is the function on $T_{\reg}$ so that at $(1 , \gm) \in M/\Delta_{T} \times T_{\reg}$, we have
\begin{equation*}
\Sha^* \omega_N= \delta_T(\gm) \omega_{M/\Delta_{T}} \wedge \omega_{T}.
\end{equation*}

  Now,
\begin{equation} \label{Guest}
\Sha^{*}\omega_{N}((1, \gamma))[(p(\underline{A_{i}},0);0);(p(0,\underline{B_j});0);(p(0,\underline{C_{k}});0);((0,0);\underline{\gamma Z_{l}})]
\end{equation}
is equal to 
$$\omega_{N}(n(\gamma))[d\Sha(p(\underline{A_{i}},0);0);d\Sha(p(0,\underline{B_{j}});0);d\Sha(p(0,\underline{C_{k}});0);d\Sha((0,0);\underline{\gamma Z_{l}})].$$

By (\ref{shadiff}), we have
\begin{equation*}
\begin{split}
d\Sha(p(\underline{A_{i}},0);0) n(\gm)^{-1}&=u(\ul{A_i} \xi,\ul{A_i} \ups^{-1}+\gm_G \ups^{-1} \ul{A_i}^*+\ul{A_i} \gm_G \ups^{-1}), \\
d\Sha(p(0,\underline{B_{j}});0)n(\gm)^{-1}&=u(-\xi \ul{B_j},0)=u(\ul{\beta_j^{*}},0),\\
d\Sha(p(0,\underline{C_{k}});0)n(\gm)^{-1}&=u(-\xi \ul{C_k},-\xi \ul{C_k} \xi^*), \text{ and}\\
d\Sha((0,0);\underline{\gamma Z_{l}})n(\gm)^{-1}&=u\left(0,-\gm_G^2 \Xi(\gm \underline{Z_{l}}) \ups^{-1}\right).\\
\end{split}
\end{equation*}
(Since $B_j \xi^*$ has image in $Y^{\perp}$, we have $\xi B_j \xi^*=0$.) Thus the map $d\Sha(\cdot)n(\gm)^{-1}$ takes the ordered basis $p\mathfrak B$ of $\mathfrak m/\Delta_{\mathfrak t} \oplus \mathfrak t$ to the ordered basis $\mathfrak B_{2}(\gm)$ of $\nn$.

Therefore (\ref{Guest}) is simply
\begin{equation*}
\begin{split}
(\Sha^* \omega_N)((1, \gamma)) [p\mathfrak B] &= \omega_N(1)[\mathfrak B_2 (\gm)] \\
						 &=  \omega_N(1)[L(\mathfrak B_1)] \\
 						&= (\det L(\gm))\omega_{N}(1)[\mathfrak B_1]\\
						&=\pm ( \det L(\gm)) \cdot \omega_{M/\Delta_{T} \times T}(1 , \gamma)[p \mathfrak B].\\
\end{split}
\end{equation*}						
 The result follows.
 \end{proof}

The fact that $\delta_T$ is nowhere vanishing has an important corollary. 
  
\begin{cor} \label{immersion}
 The map $\Sha_{T}$ is \'{e}tale, and $N^{Y,T}$ is an open subset of $N$.
\end{cor}
  
\begin{proof} 
By Corollary \ref{dimension_formula}, the dimensions of the manifolds $M/\Delta_T \times T_{\reg}$ and $N$ agree.
By Proposition \ref{delta/det}, the map $\Sha_{T}$ is \'{e}tale at $(1 , \gm) \in M/\Delta_T \times T_{\reg}$.  By (\ref{trans_diagram}), it follows that $\Sha_{T}$ is \'{e}tale at all points.
It is therefore an open map.
\end{proof}

\subsection{Choice of differential forms}

We now pin down differential forms on $M/\Delta_{T} \times T_{\reg}$ and $N$. The exact sequences from Section \ref{4X} give a natural way to build both of these forms from the same pieces. 
      
Choose left-invariant differential forms $\omega_{G}$, $\omega_{H}$, $\omega_{T}$ and $\omega_{Z_{H}(T)}$ on the groups $G, H,T$, and $Z_H(T)$.
When convenient, we will simply write $\omega_G$ for $\omega_G(1)$ at the identity of $G$, and similarly for other groups.  Note that specifying an invariant differential form at $1$ prescribes its values on the entire group.
Also fix an alternating form $\omega_{(W',Y^{\perp})}$ of top degree on $\textrm{Hom}(W',Y^{\perp})$.
These five choices will determine all the forms that we want. 

Write $\omega_{M}$ for the product of $\omega_G$ and $\omega_H$ on $M$. Next, using the exact sequence in (\ref{Z_seq}), define $\omega_{H_{Y^{\perp}}}$, a left invariant differential form on $H_{Y^{\perp}}$, so that $\omega_{Z_{H}(T)}=\omega_{H_{Y^{\perp}}}\cap\omega_{T}$. Using this and the exact sequence (\ref{kappa_seq}), we define a top degree alternating form  $\omega_{\kappa}=\omega_{(W',Y^{\perp})}\cap\omega_{H_{Y^{\perp}}}$.
We now define a form of top degree $\omega_{\mathfrak n_{\xi}}$ on $\mathfrak  n_{\xi}$ using the exact sequence (\ref{M_seq}), so that $\omega_{M}=\omega_{\mathfrak n_{\xi}}\cap\omega_{\kappa}$. 
Using the exact sequence (\ref{N_seq}) and the definition of $\omega_{\mathfrak n_{\xi}}$, we define $\omega_{N}$ on $\mf n$ as $\omega_{(W',Y^{\perp})}\cap \omega_{\mathfrak n_{\xi}}$. Extend $\omega_N$ to $N$ by left invariance as usual.  (In contrast, $\mf n_\kappa$ and $\mf n_\xi$ are only defined ``at the identity''.) % but  only at the identity element of the group $N$ but due to its translation invariance it is sufficient for them to be defined on the entire group.

Using the short exact sequence:
$$0 \to z_{\mathfrak h}(\mathfrak t)\overset{\Delta} \to \mathfrak m \overset{p} \to \mathfrak m/\Delta_{\mathfrak t} \to 0$$
(where $\Delta(z)=(\Xi(z),z)$) we define $\omega_{M/\Delta_{T}}$ at $1\cdot \Delta_{T}$ such that 
$$\omega_{M}=\omega_{M/\Delta_{T}}\cap \omega_{Z_{H}(T)}$$
and extend its value to the entire coset space via left translation. Also, denote by $\omega_{M/\Delta_{T} \times T}$ the form $\textrm{pr}_{1}^{*}(\omega_{M/\Delta_{T}})\wedge \textrm{pr}_{2}^{*}(\omega_{T})$ where $\textrm{pr}_1, \textrm{pr}_2$ are the obvious projections to $M/\Delta_{T}$ and $T$.  

Recall the choice of basis from the previous section.  The product
 \begin{equation*}
\omega_{M}[\mathfrak B]\omega_{T}(\underline{Z_{l}})=\omega_{M}[ (\underline{A_{i}},0),(0,\underline{B_{j}}),(\Xi(\underline{B'_{j'}}),\underline{B'_{j'}}),(0,\underline{C_{k}}),(\Xi(\underline{Z_{l}}),\underline{Z_{l}}) ]\omega_{T}(\underline{Z_{l}})
\end{equation*}
is equal, up to a sign, to
\begin{equation*}
 \omega_{M/\Delta_{T}\times T}(1,\gamma)[ (p(\underline{A_{i}},0);0);(p(0,\underline{B_{j}});0);(p(0,\underline{C_{k}});0);((0,0);\underline{\gamma Z_{l}}) ]\cdot {\omega_{Z_{H}(T)}(\underline{B'_{j'}},\underline{Z_{l}})}.
\end{equation*}

Consider the basis $\mathfrak B_{1}$ defined above. Clearly, $\{u(\underline{A_{i}}\xi,0),u(0,\xi\underline{C_{k}}\xi^{*}),u(0,\xi\underline{Z_{l}}\xi^{*})\}$ is a basis of $\mathfrak n_{\xi}$ while $\{\psi(u(-\xi B_{j},0))\}$ is a basis of $\textrm{Hom}(W',Y^{\perp})$. Here, for simplicity, we write `$x\doteq y$' if $x = \pm y$. We have
\begin{equation}\label{eqn_sgn_1}
\begin{split}
\omega_{M}[\mathfrak B] & \doteq \omega_{\mathfrak n_{\xi}}[u(\underline{A_{i}}\xi,0),u(0,\xi\underline{C_{k}}\xi^{*}),u(0,\xi\underline{Z_{l}}\xi^{*})]\omega_{\kappa}(\underline{B_{j}},\underline{B'_{j'}})\ ({\rm using}\ (\ref{M_seq}))\\
& \doteq\omega_{\mathfrak n_{\xi}}[u(\underline{A_{i}}\xi,0),u(0,\xi\underline{C_{k}}\xi^{*}),u(0,\xi\underline{Z_{l}}\xi^{*})]\omega_{(W',Y^{\perp})}(\underline{B_{j}\xi^{*}})\omega_{H_{Y^{\perp}}}(\underline{B'_{j'}})\ ({\rm using}\ (\ref{kappa_seq}))\\
& \doteq \omega_{N}[\mathfrak B_{1}]\omega_{H_{Y^{\perp}}}(\underline{B'_{j'}}) \ ({\rm using}\ (\ref{N_seq})).
\end{split}
\end{equation}
Using (\ref{eqn_sgn_1}) we get
\begin{equation}\label{eqn_sgn_2}
\begin{split}
 \omega_{M/\Delta_{T}\times T}(1,\gamma)[p \mathfrak B]& = \omega_{M/\Delta_{T}\times T}(1,\gamma)[(p(\underline{A_{i}},0);0);(p(0,\underline{B_j});0);(p(0,\underline{C_{k}});0);((0,0);\gm\underline{ Z_{l}})]\\
& \doteq \omega_{N}[\mathfrak B_{1}]\omega_{H_{Y^{\perp}}}(\underline{B'_{j'}})\frac{\omega_{T}(\underline{Z_{l}})}{\omega_{Z_{H}(T)}(\underline{B'_{j'}},\underline{Z_{l}})} \ ({\rm using}\ (\ref{Z_seq}))\\
& = \omega_{N}[\mathfrak B_{1}].
\end{split}
\end{equation}

\begin{remark}
We omit the calculation of the precise signs in (\ref{eqn_sgn_1}) and (\ref{eqn_sgn_2}), since it is only the associated measures that we require.  They depend only on the dimensions of the various groups.
\end{remark}

\section{Jacobian for the symplectic and orthogonal cases} \label{routes}

It remains to compute the determinant of $L(\gm)$.  In this section we treat the symplectic and orthogonal cases, and in the next we treat the unitary case.

First we explicitly calculate $\det L(\gm)$ in the ``split'' cases of Section \ref{ping} or \ref{pong}. Thus $H_{Y}$ is symplectic or split orthogonal, and we have already described the tori $T$ and $T_G$, and the characters $\chi_{i}$ of $T_G$.  At the end of this section, we prove that the same formula works for general $T \leq H_Y$.

\subsection{Root vectors} \label{routes_1}
Let $R(G,T_{G})$ and $R(H_Y,T)$ denote the set of roots of $G$ relative to $T_{G}$ and that of $H_Y$ relative to $T$ respectively. Write $R^\theta=R(G,T_{G})^\theta$ for the fixed points of $R(G,T_{G})$ under $\theta$, and $R_0=R(G,T_{G})_0$ for the complement of $R(G,T_{G})^\theta$ in $R(G,T_{G})$.  Write $\ol R_0$ for the set of $\theta$-orbits $\{ \alpha,\theta(\alpha) \}$ in $R_0$.

Consider the decomposition of $\mf n$ under the action of $T$ via $\Ad (\Delta(T))$, and write
\beq
\mf n_\beta= \{ u \in \mf n \mid \Ad(\Delta(\gm))u=\beta(\gm) u \},
\eeq
for $\beta \in \Hom(T,\mb G_m)$.  (Note that $\mf n_\beta=0$ unless $\beta=0$ or $\beta \in R(H_Y,T)$.) 
 
Then we have
\beq
\mf n=  \bigoplus_\beta \mf n_\beta,
\eeq
with
 \beq
\begin{split}
\mf n_0 &= \{ u \in \mf n \mid \Ad(\Delta(\gm))u= u \} \\
 &=\Span \{ u(A \xi,0), u(\ul {\beta_j^*}, 0), u(0,\xi \ul {Z_l} \xi^*) \mid A \in \mf t_G \}. \\
\end{split}
\eeq 
 
Write
\begin{equation*}
\res: R(G,T_{G}) \to X^*(T),
\end{equation*}
or $\alpha \mapsto \alpha_{\res}$, for the map defined by
\beq
\alpha_{\res}=\beta \Leftrightarrow \alpha|_{{}^\xi T}={}^\xi \beta.
\eeq

It is easy to determine the image and fibres of $\res$, since $\Xi$ furnishes an isomorphism $H \to G^\theta$, reducing us to computing the fibres of 
\beq
R(G,T_G) \to R((G^\theta)^\circ,(T^\theta)^\circ).
\eeq

\begin{lemma}
\begin{enumerate}
\item In the symplectic case, the map $\res$ maps $R(G,T_{G})$ onto $R(H_Y, T)$.  Let $\beta \in R(H_Y,T)$.  If $\beta$ is a long root, its fibre is a singleton in $R^\theta$.  If $\beta$ is a short root, then its fibre 
consists of a $\theta$-orbit of roots in $R_0$.
\item In the orthogonal case, the map $\res$ maps $R_0$ onto $R(H_Y,T)$.  If $\alpha \in R^\theta$, then $\alpha_{\res} \notin R(H_Y, T)$.  The fibres over $R(H_Y, T)$ are $\theta$-orbits of roots in $R_0$.
\end{enumerate}
\qed
\end{lemma}

Let $\alpha \in R(G,T_{G})$.  If $A_{\alpha}$ is a root vector for $\alpha$, then $\tau(A_{\alpha})$ is a root vector for $\theta(\alpha)$.  We will normalize our root vectors so that if $\{ \alpha, \theta(\alpha) \}$ is an orbit in $R_0$, then 
$A_{\theta(\alpha)}=\tau(A_{\alpha})$.  Note in the symplectic case that if $\alpha \in R^\theta$, then $\tau(A_{\alpha})=-A_{\alpha}$.
 
\begin{defn} If $\alpha$ corresponds to the root $\chi_{i}-\chi_{j}$ of $T_{G}$ in $\GL(W)$,  write $\lambda^{+}_{\alpha}:=\chi_i$ and $\lambda^{-}_{\alpha}:=\chi_j$. \end{defn}

Note that, for all $\alpha \in R(G,T_G)$, and $\delta \in T_G$, we have
\begin{equation} \label{functional.equation}
\lambda^-_{\alpha}(\delta)=\lambda^+_{\theta(\alpha)}(\tau(\delta)).
\end{equation}

Recall that for $A \in \mf g$, we have
\beq
L(u(A \xi,0))=u(A \xi,  A \ups^{-1}+ \gm_G \ups^{-1}A^{*} + A \gm_G \ups^{-1} ).
\eeq

\begin{prop}\label{prop_L_des} 
Suppose $H_Y$ is orthogonal.  Let $\alpha \in R^\theta$. Fix $A_\alpha \in \mf g_\alpha$.  Then 
\beq
L(u(A_{\alpha} \xi,0))=u(A_{\alpha} \xi,0).
\eeq
\end{prop}
 
\begin{proof} 
Note that $\tau(A_\alpha)=A_\alpha$ in this case.  Thus
\beq
\begin{split}
 A_\alpha \ups^{-1}+ \gm_G \ups^{-1}A_\alpha^{*} + A_\alpha \gm_G \ups^{-1} &= (A_\alpha+\gm_G \tau(A_\alpha)  + \lambda_\alpha^-(\gm_G) A_\alpha) \ups^{-1} \\
 							&= (A_\alpha+\lambda_\alpha^+(\gm_G) A_\alpha  + \lambda_\alpha^-(\gm_G) A_\alpha) \ups^{-1} \\
							&=(1+ \lambda_\alpha^+(\gm_G)+\lambda_\alpha^+(\tau(\gm_G)))A_\alpha \ups^{-1} \\
							&= 0.\\
							\end{split}
							\eeq
We have used the identity $1+ \gm_G+\tau(\gm_G)=0$ for the last equality.
\end{proof}
 
\begin{prop} Suppose  $H_Y$ is symplectic.  Let $\alpha \in R^\theta$, and fix $A_\alpha$ in $\GG$. Put $\beta=\alpha_{\res}$ and $C_\beta=\xi^{+}A_{\alpha}\xi$.
  Then $\mathfrak n_{\beta}$ is  two-dimensional and given by
  \beq
  \mf n_\beta=\Span \{ u(A_{\alpha}\xi,0), u(0,\xi C_{\beta}\xi^{*})\}. 
   \eeq
Moreover $L$ maps $\nn_{\beta}$ to $\nn_{\beta}$.  Writing $L_{\beta}$ for this restriction, we have
\begin{equation*}
\det L_{\beta}=\lambda_\alpha^-(\gm_G)-\lambda_\alpha^+(\gm_G).
\end{equation*}
\end{prop}
 \begin{proof}
By a calculation similar to the one in the proof of Proposition \ref{prop_L_des}, we have 
\begin{equation*}
L_{\beta}(u(A_{\alpha}\xi,0))=u(A_{\alpha}\xi,0)-2\lambda_{\alpha}^+( \gm_G)u(0,\xi C_{\beta}\xi^{*}).
\end{equation*}  
We further have
\begin{equation*}
L_{\beta}( u(0,\xi C_{\beta}\xi^{*})) =  -u(A_{\alpha}\xi,0)-u(0,\xi C_{\beta}\xi^{*}),
\end{equation*}
and the statement follows.
\end{proof}

\begin{defn} 
Let $\alpha \in R_0$ and $\alpha' =\theta(\alpha)$.  Put $\beta=\alpha_{\res}$.  Put $A_{\beta}=A_{\alpha}-A_{\alpha'}$, and $C_\beta=\xi^{+}A_{\beta}\xi$.
\end{defn}
Then $A_{\beta}$ is a root vector for $\beta$ in $\mathfrak g^{\theta}$, and $C_\beta$ is a root vector for $\xi^{+}(\beta)\xi$ in $\mathfrak h_{Y}$. The sign of $C_{\beta}$ depends on the choice of the $\alpha$ here although this fact has no bearing on the final outcome.
 
\begin{prop}
 Let $\alpha \in R_0$.  Then $\mf n_\beta$ is three-dimensional, and given by
 \beq 
\mathfrak n_{\beta}=\Span\{ u(A_{\alpha}\xi,0), u(A_{\alpha'}\xi,0), u(0,\xi C_{\beta}\xi^{*}) \}.  
\eeq
\qed 
\end{prop}

\begin{prop} \label{nalpha}
$L$ maps $\nn_{\beta}$ to $\nn_{\beta}$.  Writing $L_{\beta}$ for this restriction, we have
\begin{equation}\label{eq_nalpha}
\begin{array}{cccc}
\det L_{\beta}& = & \Big \{ &
\begin{array}{cc}
 \lambda^-_\alpha(\gm_G)-\lambda_\alpha^+(\gm_G) &  \text{if $H_Y$ is of type $C_n$ or $D_n$,}\\
\lambda^-_\alpha((\epsilon\gm)_{G})-\lambda_\alpha^+((\epsilon\gm)_{G}) & \text{for type $B_n$}. 
\end{array}\\
\end{array}
\end{equation}
 
\end{prop}

\begin{proof}
We will prove the proposition for the case when $H_{Y}$ is of type $C_n$ or $D_n$; it is similar for type $B_n$. Let us check that the element 
\begin{equation}
A_{\alpha}\ups^{-1}+ \gm_G \ups^{-1}A_{\alpha}^{*} + A_{\alpha}\gm_G \ups^{-1}
\end{equation}
is a multiple of $\xi C_{\beta} \xi^*= (A_{\alpha}-A_{\alpha'}) \ups^{-1}$.
Multiplying on the right by $\ups$ gives
\begin{equation*}
\begin{split}
A_{\alpha}+ \gamma_G A_{\alpha'} + A_{\alpha}  \gamma_G &= A_{\alpha}+A_{\alpha'} \lambda_{\alpha'}^+(\gm_G)+A_{\alpha}\lambda_\alpha^-(\gm_G) \\
													&= A_{\alpha}(1+\lambda_{\alpha'}^+(\gm_G)+\lambda_{\alpha}^-( \gm_G))-\lambda_{\alpha'}^+( \gm_G) {}^\xi C_\beta.\\
\end{split}
\end{equation*}
 
Applying (\ref{functional.equation}) gives that 
\begin{equation}\label{func_eq}
1+\lambda_{\alpha'}^+(\gm_G)+\lambda_{\alpha}^-( \gm_G)=0.
\end{equation}
Thus,
\begin{equation*}
L_{\beta}(u(A_{\alpha}\xi,0))=u(A_{\alpha}\xi,0)-\lambda_{\alpha'}^+( \gm_G)u(0,\xi C_{\beta}\xi^{*}).
\end{equation*}
Similarly, 
\begin{equation*}
L_{\beta}(u(A_{\alpha'}\xi,0))=u(A_{\alpha'} \xi,0) + \lambda_{\alpha}^+( \gm_G)u(0,\xi C_{\beta}\xi^{*}).
\end{equation*}
Finally we have
\begin{equation*}
L_{\beta}( u(0,\xi C_{\beta}\xi^{*})) =  -u(A_{\alpha}\xi,0)+u(A_{\alpha'}\xi,0)-u(0,\xi C_{\beta}\xi^{*}),
\end{equation*}
and the result follows.
\end{proof}

It follows from (\ref{func_eq}) in the proof above, that the right hand side of (\ref{eq_nalpha}) is the same if $\alpha$ is replaced with $\alpha'$.
 
\subsection{Computing $\det L$ when $H_{Y}$ is of type $C_n$ or $D_n$}
Throughout this subsection $H_{Y}$ is either symplectic or even orthogonal.  Let us specify the basis $\mathfrak B$ from Section \ref{JacobianSection} more precisely.  We will take the basis $\{ A_i\}$ of $\GG$ to be the union of a basis of $\mathfrak t_{\GG}$ and a choice of root vectors $A_\alpha$  for $T_G$. As in the previous section, we will normalize our root vectors so that if $\{ \alpha, \theta(\alpha) \}$ is an orbit in $R_0$, then $A_{\theta(\alpha)}=\tau(A_{\alpha})$. Also, in the symplectic case if $\alpha \in R^\theta$, then $\tau(A_{\alpha})=-A_{\alpha}$. Next, for $\{ C_k \}$ we pick the root vectors of $T$ in $\hh_{Y}$, given by $C_\beta=\xi^{+}A_\beta\xi$ or $C_\beta=\xi^{+}A_{\alpha}\xi$ as specified in Section \ref{routes_1}.
   
\begin{prop} The quantity $\det L=\det L(\gm)$ is given by \begin{equation} \label{bigprod}
 \det \gm_G \cdot \prod_{\{\alpha \}}  (\lambda_\alpha^-(\gm_G)-\lambda_\alpha^+(\gm_G)).
\end{equation}
In the orthogonal case, the product is taken over $\ol R_0$.  In the symplectic case, the product is taken over all $\theta$-orbits in $R(G,T_{G})=\ol R_0\cup R^{\theta}$.
\end{prop} 
\begin{proof} Write $\mf n_*$ for the sum of the subspaces  $\mf n_\beta$, $u(0,\xi \mf t \xi^*)$, and the span of $u(\beta_j^*, 0)$.
In the orthogonal case, include the span of $u(A_\alpha \xi,0)$, for roots $\alpha \in R^\theta$.
Each subspace $\mf n_\beta$ and $u(0,\xi \mf t \xi^*)$ are $L$-invariant.  All the elements $u(\beta_j^*, 0)$, and $u(A_\alpha \xi,0)$ for $\alpha \in R^{\theta}$ in the orthogonal case, are fixed by $L$.
One computes that the determinant of $L$ on $u(0,\xi \mf t \xi^*)$ is $\det \gm_G$, and we have computed all determinants of $L$ on $\mf n_\beta$ above.  Therefore the determinant of $L$ on $\mf n_*$ is (\ref{bigprod}).  Now
$\mf n=\mf n_* \oplus u(\mf t_G \xi,0)$.  However since $L$ restricted to $u(\mf t_G \xi,0)$ is the identity modulo $\mf n_*$, the determinant of $L$ is still (\ref{bigprod}).
\end{proof}

Henceforth in this section we fix a set $R_{0}'$ of representatives of the $\theta$-orbits of roots in $R_0$, with the extra condition that for any $\alpha\in R_{0}$, exactly one of $\{\alpha,-\alpha\}$ lie in $R_{0}'$.
 
\begin{lemma} Let $t \in T_{G}$.  Then 
\begin{enumerate}
\item  \begin{equation*}
\prod_{\alpha \in R_{0}'}  \lambda_\alpha^+(t) \lambda_\alpha^-(t)=(\det t)^{(\dim W-2)}.
\end{equation*}
\item In the symplectic case,
\begin{equation*}
\prod_{\alpha \in R^\theta}  \lambda_\alpha^+(t) \lambda_\alpha^-(t)=(\det t)^2.
\end{equation*}
\end{enumerate} 
\qed
\end{lemma}
 
Put 
\begin{equation*}
L_0(\gm)= \prod_{\alpha \in R_{0}'}  (\lambda_\alpha^-(\gm_G)-\lambda_\alpha^+(\gm_G)),
 \end{equation*}
and
\begin{equation*}
L_\theta(\gm)= \prod_{\alpha \in R^\theta}  (\lambda_\alpha^-(\gm_G)-\lambda_\alpha^+(\gm_G)).
 \end{equation*}

Thus 
\beq
\det L(\gm)=(\det \gm_G) \cdot L_0(\gm) L_\theta(\gm)
\eeq
 in the symplectic case, and $\det L(\gm)=(\det \gm_G) \cdot L_0(\gm)$ in the orthogonal case.

\begin{lemma} 
 For $\alpha \in R(G,T_{G})$ and $\gm \in T_{\reg}$ we have
\begin{enumerate}
\item $\lambda_\alpha^+(\gm_G)-\lambda_\alpha^-(\gm_G)=(\lambda_{\alpha}^-({}^\xi \gm)-\lambda_{\alpha}^+({}^\xi \gm))\lambda_{\alpha}^+(\gm_G) \lambda_{\alpha}^-(\gm_G) .$

 \item $ \lambda_\alpha^+(\gm_{G})-\lambda_\alpha^-(\gm_{G})=\lambda_{\theta(\alpha)}^+(\gm_{G})-\lambda_{\theta(\alpha)}^-(\gm_{G}).$
 \end{enumerate}
\qed
 \end{lemma}
 
 We have
 \begin{equation*}
 \begin{split}
L_0(\gm) & =  \prod_{\alpha \in R_{0}'}  (\lambda_\alpha^-(\gm_G)-\lambda_\alpha^+(\gm_G)) \\
	& =  (\det \gm_G)^{(\dim W-2)} \prod_{\alpha \in R_{0}'} (\lambda_{\alpha}^+({}^\xi \gm)-\lambda_{\alpha}^-({}^\xi \gm)) \\
		& = (\det \gm_G)^{(\dim W-2)} \prod_{\alpha \in R_{0}'} (\alpha({}^\xi \gm)-1).\\
\end{split}
 \end{equation*}
% Here, for simplicity, we write `$x \doteq y$' if $x=\pm y$. 
Similarly,
  \begin{equation*}
 \begin{split}
L_{\theta}(\gm) & =  \prod_{\alpha \in R^\theta}  (\lambda_\alpha^-(\gm_G)-\lambda_\alpha^+(\gm_G)) \\
	& = (\det \gm_G)^2 \prod_{\alpha \in R^\theta} (\lambda_{\alpha}^+({}^\xi \gm)-\lambda_{\alpha}^-({}^\xi \gm)) \\
		&= (\det \gm_G)^2 \prod_{\alpha \in R^\theta} (\alpha({}^\xi \gm)-1).\\
\end{split}
 \end{equation*}

For $\gm \in T$, write as usual
\begin{equation} \label{disc.defn}
D_{H_Y}(\gm)=\det(\Ad(\gm)-1; \mf h_Y/\mf t).
\end{equation}
 
$D_n$ case: Regrouping gives
 \begin{equation*}
 \begin{split}
\det L(\gm) & = (\det \gm_G)^{ \dim W-1} \left(\prod_{\alpha \in R_{0}'} (\alpha({}^\xi \gm)-1)\right) \\
 	& = (\det \gm_G)^{\dim W-1} (D_{H_{Y}}(\gm)).
\end{split} 
 \end{equation*}

$C_n$ case: Regrouping gives
 \begin{equation*}
 \begin{split}
\det  L(\gm) & = (\det \gm_G)^{1+ \dim W} \left(\prod_{\alpha \in R_{0}'} (\alpha({}^\xi \gm)-1) \right) \left( \prod_{\alpha \in R^\theta} (\alpha({}^\xi \gm)-1) \right). \\
 	& =  (\det \gm_G)^{\dim W+1} D_{H_{Y}}(\gm).
\end{split} 
 \end{equation*}

%\begin{remark}
%Here and elsewhere, the omitted sign $+/-$ depends only on $\dim W$.
%\end{remark}

\begin{cor}\label{prop_spoev_spc_det} 
We have%, up to a sign depending on the constant $\dim W$,
\begin{equation*}
\delta_T(\gm)= \det L_{\mathfrak B}(\gm)=(\det \gm_G)^{\dim W \pm 1}  D_{H_{Y}}(\gm),
\end{equation*}
and 
\begin{equation*}
\Sha_{Y,T}^*(\omega_N)= \delta_N(m) (\det \gm_G)^{\dim W \pm 1}  D_{H_{Y}}(\gm) \omega_{M/\Delta_{T}} \wedge \omega_{T}
\end{equation*}
at the point $(m , \gm)$. Here the exponent is $\dim W+1$ for $H_Y$ of type $C_n$, and $\dim W-1$ for type $D_n$.
\qed
\end{cor}
 
\subsection{Computing $\det L$ when $H_{Y}$ is of type $B_n$}
The statements and the proofs in this case are similar to the even orthogonal case treated in the previous section.  As in that case, we will assume that the basis $\{ A_i\}$ of $\GG$ is the union of a basis of $\mathfrak t_{\GG}$ and a basis of root vectors $A_\alpha$  for $T_G$ (normalized as in Section \ref{routes_1}), and $\{ C_k \}$ is the basis of root vectors of $T$ in $\hh_{Y}$, consisting of  the $C_\beta$ (recall that $C_{\beta}=\xi^{+}A_\beta \xi$).

Again denote by $L_T$ the restriction of $L$ to $\xi \mf t \xi^*$; thus
\beq
L_T(  \xi Z \xi^*)= -(\epsilon \gm)_G^2 \Xi (\epsilon \gm Z) \ups^{-1}.
\eeq

\begin{lemma}  The image of $L_T$ lies in $\xi \mf t \xi^*$.  We have
\beq
\det L_T(\gm)=-2\det (\epsilon\gm)_{G}.
\eeq
Moreover the elements  $u(A \xi,0)$, for $A \in \mf t_G$, and $u(\beta_j^*, 0)$ are fixed by $L$.
\qed
\end{lemma}

\begin{prop} The quantity $\det L=\det L(\gm)$ is given by
\begin{equation} \label{bigprod_1}
-(2\det (\epsilon\gm)_G) \cdot \prod_{\{\alpha \in \ol R_0 \}}  (\lambda_\alpha^-((\epsilon\gm)_G)-\lambda_\alpha^+((\epsilon\gm)_G)).
\end{equation}
\qed
\end{prop} 

Henceforth in this section we fix a set $R_{0}'$ of representatives of the $\theta$-orbits of roots in $R_0$, with the extra condition that for any $\alpha\in R_{0}$, exactly one of $\{\alpha,-\alpha\}$ lie in $R_{0}'$.
 
\begin{lemma} Let $t={\rm diag}(t_{i}) \in T_G$.  Then 
 \begin{equation*}
\prod_{\alpha \in R_{0}'}  \lambda_\alpha^+(t) \lambda_\alpha^-(t)=(t_{\frac{\dim W+1}{2}})(\det t)^{(\dim W-2)}.
\end{equation*}
\qed
\end{lemma}
 
Put 
\begin{equation*}
L_0(\gm)= \prod_{\alpha \in R_{0}'}  (\lambda_\alpha^-((\epsilon\gm)_G)-\lambda_\alpha^+((\epsilon\gm)_G)).
 \end{equation*}
Thus $\det L(\gm)=(-2\det (\epsilon\gm)_G) L_0(\gm)$ in this case.

\begin{lemma} For $\alpha \in R(G,T_{G})$ and $\gm \in T_{\reg}$ we have
 \begin{equation*}
  \lambda_\alpha^-((\epsilon\gm)_G)-\lambda_\alpha^+((\epsilon\gm)_G)=(\lambda_{\alpha}^+({}^\xi (\epsilon\gm))-\lambda_{\alpha}^-({}^\xi (\epsilon\gm)))\lambda_{\alpha}^+((\epsilon\gm)_G) \lambda_{\alpha}^-((\epsilon\gm)_G) .
 \end{equation*}

\qed 
\end{lemma}
  
 We have
 \begin{equation*}
 \begin{split}
L_0(\gm)& =   \prod_{\alpha \in R_{0}'}  (\lambda_\alpha^-((\epsilon\gm)_G)-\lambda_\alpha^+((\epsilon\gm)_G)) \\
	& = -\frac{1}{2}(\det (\epsilon\gm)_G)^{(\dim W-2)} \prod_{\alpha \in R_{0}'} (\lambda_{\alpha}^+({}^\xi (\epsilon\gm))-\lambda_{\alpha}^-({}^\xi (\epsilon\gm))) \\
		& = -\frac{1}{2}(\det (\epsilon\gm)_G)^{(\dim W-2)} \prod_{\alpha \in R_{0}'} (\alpha({}^\xi (\epsilon\gm))-1).\\
\end{split}
 \end{equation*}

Regrouping gives
 \begin{equation*}
 \begin{split}
\det L(\gm) & = (\det (\epsilon\gm)_G)^{ \dim W-1} \left(\prod_{\alpha \in R_{0}'} (\alpha({}^\xi (\epsilon\gm))-1) \right) \\
 	& = (\det (\epsilon\gm)_G)^{\dim W-1} (D_{H_{Y}}(\epsilon\gm)).
\end{split} 
 \end{equation*}
 
 Note that  $D_{H_{Y}}(\epsilon\gm)= D_{H_{Y}}(\gm)$.

%\begin{remark} \label{Arnab.remark}
%As before, the omitted sign $+/-$ depends only on $\dim W$.
%\end{remark}

\begin{cor}\label{prop_spo_spc_det_1} 
We have
\begin{equation} \label{323}
\delta_{T}(\gm)=(\det (\epsilon\gm)_G)^{\dim W - 1}  D_{H_{Y}}(\gm),
\end{equation}
and 
\begin{equation*}
\Sha_{Y,T}^*(\omega_N)=\delta_N(m) (\det (\epsilon\gm)_G)^{\dim W - 1}  D_{H_{Y}}(\gm) \omega_{M/\Delta_{T}} \wedge \omega_{T}
\end{equation*}
at the point $(m , \gm)$.
\qed
\end{cor}

% By passing to a suitable algebraic extension, we see that (\ref{323}) holds for general $H$ and $T$.

\subsection{Symplectic or orthogonal groups: the general case}
 
We give here an argument to reduce the computation of $\det L(\gm)$ in the general (orthogonal or symplectic) case to that of the previous section.  

We are given an inner product space $(V,\Phi)$, symmetric or antisymmetric, with $V$ a vector space, isotropic subspaces $W,W'$ such that $\Phi|_{W+W'}$ is nondegenerate, a nondegenerate subspace $Y \subseteq X=(W+W')^\perp$ of dimension $k$, and a maximal torus $T \leq H_Y$.

Write $(V_0,\Phi_0)$ for the fixed inner product space of Section \ref{ping} or \ref{pong}, which is either symmetric or antisymmetric to agree with $V$.  We shall add the subscript `${}_0$' to all the constructs associated to this space such as its subspaces, the bilinear form, the tori in the isotropic groups etc.  Naturally, 
we choose the dimensions of $V_0$, $W_0$, $Y_0$ to equal the dimensions of $V,W,Y$.  

Let us write $V_{\ol F}$ for $V \otimes_F \ol F$ in this section, and similarly for the other spaces.  Then there is an isometric isomorphism
 
\beq
\varphi: (V_{\ol F}, \Phi_{\ol F}) \iso (V_{0,\ol F}, \Phi_{0 ,\ol F}),
\eeq
with $\varphi(W _{\ol F})=W_{0, \ol F}$, $\varphi(W' _{\ol F})=W'_{0, \ol F}$, $\varphi(Y _{\ol F})=Y_{0,\ol F}$, $\varphi(T(\ol F))=T_0(\ol F)$, and $\varphi(\mf t _{\ol F})=\mf t_{0,\ol F}$.

\begin{prop}\label{prop_spo_spc_det}
Suppose that $H_Y$ is of type $B_n$, $C_n$ or $D_{n}$.  For $\gm \in T_{\reg}(F)$ we have
\[
\delta_T(\gm)=
\begin{cases}
(\det (\epsilon\gm)_G)^{\dim W - 1}  D_{H_{Y}}(\gm) & {\rm if}\ H_{Y} {\rm \ is \ of \ type\ } B_n \\
(\det \gm_G)^{\dim W + 1}  D_{H_{Y}}(\gm) & {\rm if}\ H_{Y} {\rm \ is \ of \ type\ } C_n  \\
(\det \gm_G)^{\dim W - 1}  D_{H_{Y}}(\gm) & {\rm if}\ H_{Y} {\rm \ is \ of \ type\ } D_n. \\
\end{cases}
\]
\end{prop}

\begin{proof}  
Suppose first that $H_{Y}$ is of type $C_{n}$. We have
\beq
\begin{split}
\det {}_F L(\gm) &= \det{}_{\ol F} L(\gm)\\
			&= \det{}_{\ol F} ({}^\varphi L)({}^\varphi \gm) \\		
			&=  (\det ({}^\varphi \gm)_G)^{\dim W + 1}  D_{H_{Y_0}}({}^\varphi \gm) ({\rm using \ } {\rm Corollary}\ \ref{prop_spoev_spc_det})\\
			&= (\det \gm_G)^{\dim W + 1}  D_{H_{Y}}(\gm),
\end{split}
\eeq
as desired. The other two cases are dealt with in a similar manner.
\end{proof}

\section{Jacobian for the  unitary cases}\label{routes_2}
Next we treat the case of the unitary groups.

\subsection{Quasisplit $H_Y$, maximally split $T$}
Let us do an explicit calculation of $\det L(\gm)$ when we are in the situation of Section \ref{Hermitian.Case}.    Thus $H_{Y}$ is a quasisplit unitary group, and we have already described the tori $S,T,S_G$, and $T_G$ and 
characters $\chi_{i}: T_G \to E^{\times}$.  Note that $\chi_i(S_G)=F^{\times}$.  Write $N=N_{E/F}: E^\times \to F^\times$ for the usual norm map $N(x)=x \ol x$, and $\tr=\tr_{E/F}$ for the usual trace map.
 Recall that we have fixed an element $\iota\in E^{\times}$ such that $\sigma(\iota)=-\iota$.  Given an $x\in E$, define $(x)_{1}$ and $(x)_{2}$ to be the unique elements in $F$ such that $x=(x)_{1}+(x)_{2}\iota$.  
 
 Consider the decomposition of $\mf n$ under the action of $S$ via $\Ad (\Delta(S))$, and write
\beq
\mf n_\beta= \{ u \in \mf n \mid \Ad(\Delta(\gm))u=\beta(\gm) u \},
\eeq
for $\beta \in \Hom(S,\mb G_m)$.  (Note that $\mf n_\beta=0$ unless $\beta=0$ or $\beta \in R(H_Y,S)$.) 
 
Then we have
\beq
\mf n= \mf n_0 \oplus \bigoplus_{\beta\in R(H_Y,S)} \mf n_\beta,
\eeq
where 
 \beq
\begin{split}
\mf n_0 &= \{ u \in \mf n \mid \Ad(\Delta(\gm))u= u \} \\
 &=\Span \{ u(A \xi,0), u(\ul {\beta_j^*}, 0), u(0,\xi \ul {Z_l} \xi^*) \mid A \in \mf t_G \}. \\
\end{split}
\eeq

This time we record the image and fibres of the restriction map
\begin{equation*}
\res: R(G,S_{G}) \to \Hom(S,\Gm/_{F}),
\end{equation*}
which we write as $\alpha \mapsto \alpha_{\res}$, defined as before.

The sets $R(G,S_{G})$ and $R(H_Y,S)$ denote the set of roots of $G$ relative to $S_{G}$ and those of $H_Y$ relative to $S$ respectively. Write $R^\theta=R(G,S_{G})^\theta$ for the fixed points of $R(G,S_{G})$ under $\theta$, and $R_0=R(G,S_{G})_0$ for the complement of $R(G,S_{G})^\theta$ in $R(G,S_{G})$.  Write $\ol R_0$ for the set of $\theta$-orbits $\{ \alpha,\theta(\alpha) \}$ in $R_0$.

\begin{lemma}
The map $\res$ maps $R(G,S_{G})$ onto $R(H_Y, S)$.  If $\beta \in R(H_Y,S)$ is a root of the form $2 \chi_{i}$, for some $i$, then its fibre is a singleton in $R^{\theta}$. Otherwise its fibre consists of a $\theta$-orbit of roots in $R_0$.
\qed
\end{lemma}

Let $\alpha \in R(G,S_{G})$. Note that in this case the root spaces $\mathfrak g_{\alpha}$ are two-dimensional $F$-vector spaces. If $A_{\alpha}$ is a root vector for $\alpha$, then $\tau(A_{\alpha})$ is a root vector for $\theta(\alpha)$. 

\begin{defn} 
Fix $\alpha \in R_0$ and let $\alpha' =\theta(\alpha)$.  Put $\beta=\alpha_{\res}$.  
 Choose a nonzero element $A_{\alpha}^{1}\in \mathfrak g_{\alpha}$.

We define
\begin{itemize}
\item $A_{\alpha}^{2}=\iota A_{\alpha}^{1}$,
\item $A_{\alpha'}^{i}=\tau(A_{\alpha}^{i})$ for $i=1,2$,
\item $A_{\beta}^{1}=A_\alpha^{1}-A_{\alpha'}^{1}$, 
\item $A_{\beta}^2=\iota(A_\alpha^{1}+A_{\alpha'}^{1})$, and
\item $C_\beta^{i}=\xi^{+} A_{\beta}^{i}\xi$ for $i=1,2$. 
 \end{itemize}
 %(see Definition \ref {xidef}).
 \end{defn}

The set $\{A_{\alpha}^{1},A_{\alpha}^{2}\}$ is then an $F$-basis of $\mathfrak g_{\alpha}$, and similarly for $\alpha'$.  Note that $A_{\alpha'}^2=-\iota A^1_{\alpha'}$.
Also, $A_{\beta}^{i}= A_\alpha^{i}-A_{\alpha'}^{i} $ is a root vector for $\beta$ in $\GG^{\theta}$, and $C_\beta^{i}$ is a root vector for $\beta$ in $\hh_{Y}$. 

\begin{defn}
Let $\alpha\in  R^{\theta}$ and $\beta=\alpha_{\res}$. Let $C_{\beta}$ be a root vector in $\mathfrak h_{Y}$ for $\beta$. (Note that the root space for $\beta$ in $\mathfrak h_{Y}$ is of dimension one.) Put $A^{1}_{\alpha}=\Xi (C_{\beta})$ and $A^{2}_{\alpha}=\iota A^{1}_{\alpha}$.
\end{defn}
Clearly $A^{1}_{\alpha}$ and $A^{2}_{\alpha}$ are root vectors for $\alpha$ in $\mathfrak g$ and form an $F$-basis of $\mathfrak g_{\alpha}$.

\begin{prop}
\begin{enumerate}
 \item Fix $\alpha \in R_0$ as above, with $\alpha_{\res}=\beta$. Then $\mathfrak n_{\beta}$ is a six-dimensional $F$-space, given by
 \beq
\mf n_\beta =\Span_F \{ u(A_{\alpha}^{i}\xi,0), u(A_{\alpha'}^{i}\xi,0), u(0,\xi C_{\beta}^{i}\xi^{*}) \mid i=1,2\}.  
 \eeq
\item  Let $\alpha \in R^{\theta}$.  Then $\mathfrak n_{\beta}$ is a three-dimensional $F$-space, given by
\beq
\mf n_\beta= \Span_F \{ u(A_{\alpha}^{1}\xi,0), u(A_{\alpha}^{2}\xi,0), u(0,\xi C_{\beta}\xi^{*})\}. 
\eeq
\end{enumerate}
 If $\alpha$ is the root $\chi_{i}-\chi_{j}$,  put $\lambda^{+}_{\alpha}:=\chi_i$ and $\lambda^{-}_{\alpha} :=\chi_j$.
\qed
\end{prop}

Note that for all $\delta \in T_G$, we have
\begin{equation} \label{brexit}
\lambda_{\alpha}^-(\tau(\delta))=\lambda_{\theta(\alpha)}^+(\ol \delta),
\end{equation}
and
\beq
\delta A_\alpha = \lambda^{+}_{\alpha}(\delta) A_{\alpha}, \text{  } A_\alpha \delta =  \lambda^{-}_{\alpha}(\delta) A_{\alpha}.
\eeq

\begin{prop} \label{nalpha.unitary}
$L$ maps $\nn_{\beta}$ to $\nn_{\beta}$.  Writing $L_{\beta}$ for this restriction, when $\alpha \in R_0$, we have
\begin{equation*}
\det L_{\beta}= N \big(1+\lambda_{\alpha}^{+}(\gm_G)+\ol{\lambda_{\alpha'}^{+}(\gm_G)}\big).
\end{equation*}
 When $\alpha \in R^{\theta}$ we have
\begin{equation*}
\det L_{\beta}= -1-{\rm tr}(\lambda_{\alpha}^{+}(\gm_G)).
\end{equation*}
\end{prop}

\begin{proof}
We first treat the case when $\alpha \in R_0$. Let us begin by checking that for each $i$ the elements 
\begin{equation*}
A_{\alpha}^{i}\ups^{-1}+ \gm_G \ups^{-1}(A_{\alpha}^{i})^{*} + A_{\alpha}^{i}\gm_G \ups^{-1}
\end{equation*}
lie in the span of $\xi C_{\beta}^{1} \xi^*$ and $\xi C_{\beta}^{2} \xi^*$. 
Using (\ref{brexit}) and the fact that $1+ \gm_G+\tau(\gm_{G})=0$, we get that
\begin{equation*}
\begin{split}
(A_{\alpha}^{i}\ups^{-1}+ \gm_G \ups^{-1}(A_{\alpha}^{i})^{*} + A_{\alpha}^{i}\gm_G \ups^{-1})\upsilon&=A_{\alpha}^{i}+ \gamma_G A_{\alpha'}^{i} + A_{\alpha}^{i}  \gamma_G \\													
													&= A_{\alpha}^{i}+\lambda_{\alpha'}^+(\gm_G)A_{\alpha'}^{i} +\lambda_\alpha^-(\gm_G)A_{\alpha}^{i} \\
													&= -\lambda_{\alpha}^-( \tau(\gm_G))A_{\alpha}^{i}+ \overline{\lambda_{\alpha}^{-}(\tau(\gm_G))} A_{\alpha'}^{i}\\
													& =-(\lambda_{\alpha}^-( \tau(\gm_G)))_{1}(A_{\alpha}^{i}-A_{\alpha'}^{i})-(\lambda_{\alpha}^-( \tau(\gm_G)))_{2}\iota(A_{\alpha}^{i}+A_{\alpha'}^{i}).\\
\end{split}
\end{equation*}
For $i=1$ this is equal to
\begin{equation*}
-(\lambda_{\alpha}^-( \tau(\gm_G)))_{1} \cdot {}^\xi C_{\beta}^{1}-(\lambda_{\alpha}^-( \tau(\gm_G)))_{2} \cdot {}^\xi C_{\beta}^{2},
\end{equation*}
while for $i=2$ this is equal to
\begin{equation*}
-(\lambda_{\alpha}^-( \tau(\gm_G)))_{1} \cdot {}^\xi  C_{\beta}^{2}-\iota^{2}(\lambda_{\alpha}^-( \tau(\gm_G)))_{2} \cdot {}^\xi C_{\beta}^{1}.
\end{equation*}
This proves the first statement.

Thus,
\begin{equation*}
L_{\beta}(u(A_{\alpha}^{1}\xi,0))=u(A_{\alpha}^{1}\xi,0)-(\lambda_{\alpha}^-( \tau(\gm_G)))_{1}u(0,\xi C_{\beta}^{1}\xi^{*})-(\lambda_{\alpha}^-( \tau(\gm_G)))_{2}u(0,\xi C_{\beta}^{2}\xi^{*})
\end{equation*}
and 
\begin{equation*}
L_{\beta}(u(A_{\alpha}^{2}\xi,0))=u(A_{\alpha}^{2}\xi,0)-\iota^{2}(\lambda_{\alpha}^-( \tau(\gm_G)))_{2}u(0,\xi C_{\beta}^{1}\xi^{*})-(\lambda_{\alpha}^-( \tau(\gm_G)))_{1}u(0,\xi C_{\beta}^{2}\xi^{*}).
\end{equation*}
Similarly, 
\begin{equation*}
 L_{\beta}(u(A_{\alpha'}^{1}\xi,0))=u(A_{\alpha'}^{1}\xi,0)+(\lambda_{\alpha'}^-( \tau(\gm_G)))_{1}u(0,\xi C_{\beta}^{1}\xi^{*})-(\lambda_{\alpha'}^-( \tau(\gm_G)))_{2}u(0,\xi C_{\beta}^{2}\xi^{*})
\end{equation*}
and 
\begin{equation*}
L_{\beta}(u(A_{\alpha'}^{2}\xi,0))=u(A_{\alpha'}^{2}\xi,0)-\iota^{2}(\lambda_{\alpha'}^-( \tau(\gm_G)))_{2}u(0,\xi C_{\beta}^{1}\xi^{*})+(\lambda_{\alpha'}^-( \tau(\gm_G)))_{1}u(0,\xi C_{\beta}^{2}\xi^{*}).
\end{equation*}
Finally we have
\begin{equation*}
L_{\beta}( u(0,\xi C_{\beta}^{i}\xi^{*})) =  -u(A_{\alpha}^{i}\xi,0)+u(A_{\alpha'}^{i}\xi,0)-u(0,\xi C_{\beta}^{i}\xi^{*}), \ i=1,2
\end{equation*}
and the result follows from a determinant calculation.

Next assume that $\alpha \in R^{\theta}$. Note that $\tau(A_{\alpha}^{1})=-A_{\alpha}^{1}$ and $\tau(A_{\alpha}^{2})=A_{\alpha}^{2}$. We have 
\begin{equation*}
\begin{split}
(A_{\alpha}^{1}\ups^{-1}+ \gm_G \ups^{-1}(A_{\alpha}^{1})^{*} + A_{\alpha}^{1}\gm_G \ups^{-1})\upsilon&=A_{\alpha}^{1}(1+ \lambda_\alpha^-(\gm_G))+\gamma_G \tau(A_{\alpha}^{1}) \\						
													&= -\lambda_{\alpha}^-( \tau(\gm_G))A_{\alpha}^{1}- \lambda_{\alpha}^{+}(\gm_G)A_{\alpha}^{1}\\
													& = -2(\lambda_{\alpha}^{+}(\gm_G))_{1}A_{\alpha}^{1}s.\\
\end{split}
\end{equation*}
Similarly, 
\begin{equation*}
(A_{\alpha}^{2}\ups^{-1}+ \gm_G \ups^{-1}(A_{\alpha}^{2})^{*} + A_{\alpha}^{2}\gm_G \ups^{-1})\upsilon =2\iota(\lambda_{\alpha}^{+}(\gm_G))_{2}A_{\alpha}^{2}. 								 
\end{equation*}
Thus,
\begin{equation*}
L_{\beta}(u(A_{\alpha}^{1}\xi,0))=u(A_{\alpha}^{1}\xi,0)-2(\lambda_{\alpha}^{+}(\gm_G))_{1}u(0,\xi C_{\beta}\xi^{*}),
\end{equation*}
\begin{equation*}
L_{\beta}(u(A_{\alpha}^{2}\xi,0))=u(A_{\alpha}^{2}\xi,0)+2\iota^{2}(\lambda_{\alpha}^{+}(\gm_G))_{2}u(0,\xi C_{\beta}\xi^{*}),
\end{equation*}
and
\begin{equation*}
L_{\beta}( u(0,\xi C_{\beta}\xi^{*})) = -u(A_{\alpha}^{1}\xi,0)-u(0,\xi C_{\beta}\xi^{*})
\end{equation*}
and the result follows in this case as well.
\end{proof}

\subsection{Computing $\det L$ in the unitary case}
Let us take the basis $\{ A_i\}$ of $\GG$ to be the union of a basis of $\Lie(T_G)$  and a basis of root vectors $A_\alpha$  for $S_G$, normalized as above.  Next, we choose for $\{ C_k \}$ the basis of root vectors of $S$ in $\hh_{Y}$, consisting of $C_{\beta}$ such that ${}^{\xi}C_\beta=A_\beta$ (when $\beta$ is not a root of the form $2\chi_{i}$ for any $i$), or ${}^\xi C_{\beta}=A_{\alpha}^{1}, A_{\alpha}^{2}=\iota A_{\alpha}^{1}$ (when $\beta$ is a root of the form $2\chi_{i}$ for some $i$), as specified in the previous section.  Again denote by $L_T$ the restriction of $L$ to $\xi \mf t \xi^*$.

\begin{lemma}  The image of $L_T$ lies in $\xi \mf t \xi^*$.  We have
\beq
\det L_T(\gm)=N (\det(\gm_G)).
\eeq
\qed
\end{lemma} 
Moreover the elements  $u(A \xi,0)$, for $A \in \mf t_G$, and $u(\ul {\beta_j^*}, 0)$ are fixed by $L$.

\begin{prop} \label{bigprod_unitary1}
The quantity $\det L=\det L(\gm)$ is given by
\begin{equation} \label{bigprod_unitary}
 N ({\rm det}_{E} (\gm_G)) \ \prod_{\alpha \in R^\theta }(1+{\rm tr}(\lambda_{\alpha}^{+}(\gm_G))) \prod_{\alpha \in \ol R_0 } N \big( 1+\lambda_{\alpha}^{+}(\gm_G)+\ol{\lambda_{\alpha'}^{+}(\gm_G)}\big). 
\end{equation}
 
\qed
\end{prop} 
 
Let $\gm=\diag(t_{1},\dots,t_{r},c,\overline{t_{r}}^{-1}, \dots,\overline{t_{1}}^{-1})$, with $ N (c)=1$, or $\diag(t_{1},\dots,t_{r},\overline{t_{r}}^{-1}, \dots,\overline{t_{1}}^{-1})$ depending on whether $\dim W$ is odd or even.  

\begin{lemma}\label{uni_dis_1}
\begin{equation*}
\prod_{\alpha\in R^{\theta}} \left( 1+{\rm tr}(\lambda_{\alpha}^{+}(\gm_G)) \right)=\prod_{i=1}^{r} \frac{N(t_i)-1}{N (t_{i}-1)} \cdot \frac{N(t_i^{-1})-1}{N(t_{i}^{-1}-1)}.
\end{equation*}
\end{lemma}

\begin{proof} 
Let $\alpha=\chi_i-\chi_{n+1-i} \in R^{\theta}$. It is easy to check that 
\begin{equation*}
\left(1+{\rm tr}(\lambda_{\alpha}^{+}(\gm_G) \right) \left( 1+{\rm tr}(\lambda_{-\alpha}^{+}(\gm_G)) \right)=-\left(\half {\rm tr}\left(\frac{t_{i}+1}{t_{i}-1}\right) \right)^{2}.
\end{equation*}
Now using the identity
\begin{equation*}
\begin{split}
{\rm tr}\left(\frac{t+1}{t-1}\right) &=2 \cdot \frac{N(t)-1}{ N(t-1)} \\
							&=-2 \cdot \frac{N(t^{-1})-1}{N(t^{-1}-1)},\\
							\end{split}
\end{equation*}
we get the result.
\end{proof}

Henceforth in this section we fix a set $R_{0}'$ of representatives of the $\theta$-orbits of roots in $R_0$, with the extra condition that for any $\alpha\in R_{0}$, exactly one of $\{\alpha,-\alpha\}$ lie in $R_{0}'$.

\begin{lemma}\label{uni_dis_2}
The product $\prod_{\alpha \in R_{0}'} N \big( 1+\lambda_{\alpha}^{+}(\gm_G)+\ol{\lambda_{\alpha'}^{+}(\gm_G)}\big)$ is equal to

\begin{equation} \label{term.I}
 \prod_{1 \leq i<j \leq r} N \left( \frac{t_i \ol t_j-1}{(t_i-1)(\ol t_j -1)} \right)^2 N \left( \frac{t_i - t_j}{(t_i-1)( t_j -1)} \right)^2 
 \end{equation}
 when $\dim_E W$ is even, and equal to the product of (\ref{term.I}) with
 \beq
 \prod_{1 \leq i \leq r} N \left( \frac{t_i -c}{(t_i-1)(c -1)} \right)^2
 \eeq
 when $\dim_E W$ is odd.
\qed
\end{lemma}

Recall  the discriminant $D_{H_{Y}}(\gm)=\det_{F}(\Ad(\gm)-1;\mathfrak h_{Y}/\mathfrak t)$.  One computes that when $\dim_E W$ is even, $D_{H_{Y}}(\gm)$ is the product of
\begin{equation} \label{term.A}
\prod_{i=1}^r (N(t_i)-1)(N(t_i^{-1})-1)
\end{equation}

and

\begin{equation} \label{term.B}
\prod_{1 \leq i<j \leq r} N \left( \left(\frac{t_i}{t_j}-1 \right)   \left(\frac{t_j}{t_i}-1 \right) (t_i \ol t_j-1) \left( \frac{1}{t_i \ol t_j}-1\right) \right).  
\end{equation}

When $\dim_E W$ is odd,  $D_{H_{Y}}(\gm)$ is the product of (\ref{term.A}), (\ref{term.B}), and
\beq
\prod_{i=1}^r N \left( \left( \frac{c}{t_i}-1 \right) \left( \frac{t_i}{c}-1 \right) \right).
\eeq

\begin{prop}\label{prop_unit_spc_det} 
We have
\beq
\det L(\gm)= (-1)^{[(\dim_{E} W) / 2]}N ({\rm det}_{E} (\gm_G))^{\dim_{E} W}  D_{H_{Y}}(\gm)
\eeq
(where as usual $[x]$ denotes the greatest integer less than or equal to $x$).
\end{prop}

\begin{proof}
Upon dividing out $N(\det_E \gm_G)^{\dim_E W}$ from (\ref{bigprod_unitary}) and using Lemmas \ref{uni_dis_1} and \ref{uni_dis_2}, we are left with the product of  
\begin{equation} \label{term.X}
\prod_{i=1}^{r} \frac{(N(t_i)-1)(N(t_i^{-1})-1)}{N(t_i)^{2(r-1)}}
\end{equation}
with
\begin{equation} \label{term.Y}
\prod_{1 \leq i,j \leq r} N(t_i \ol t_j-1)^2 N(t_i-t_j)^2
\end{equation}
when $\dim_E W$ is even.  But this is simply $D_{H_{Y}}(\gm)$.  The case where $\dim_E W$ is odd is similar.
\end{proof}

Since it depends only on the constant $\dim_{E} W$, henceforth we ignore the sign in the expression of $\det L(\gm)$.

\begin{cor} 
We have
\begin{equation*}
\delta_T(\gm)=  N({\rm det}_{E} (\gm_G))^{\dim_E W} D_{H_{Y}}(\gm),
\end{equation*}
and 
\begin{equation*}
\Sha_{Y,T}^*(\omega_N)=\delta_N(m)  D_{H_{Y}}(\gm)( N{\rm det}_{E} (\gm_G))^{\dim_E W}  \omega_{M/\Delta_{T}} \wedge \omega_{T}.
\end{equation*}
\qed
\end{cor}

\subsection{Unitary groups: the general case}

We sketch here an argument to reduce the computation of $\det L(\gm)$ for arbitrary unitary groups to that of the previous section.  
We are given a Hermitian space $(V,\Phi)$, with $V$ an $E$-vector space, isotropic $W,W'$, a nondegenerate subspace $Y \subseteq X=(W+W')^\perp$ of dimension equal to that of $W$, and a maximal torus $T < H_Y$.

Write $(V_0,\Phi_0)$ for the fixed Hermitian space of Section \ref{Hermitian.Case}, and $W_0, W_0', Y_0, G_0, L_0,$ and $T_0$ for the corresponding data for $V_0$.  Naturally, we choose the dimensions of $V_0$, $W_0$, $Y_0$ to equal the dimensions of $V,W,Y$.  Let us write $V_{\ol F}$ for $V \otimes_F \ol F$ in this section, and similarly for other spaces.

 Then there is an isometric isomorphism

\beq
\varphi: (V_{\ol F}, \Phi_{\ol F}) \iso (V_{0, \ol F}, \Phi_{0 ,\ol F}),
\eeq
with $\varphi(W _{\ol F})=W_{0, \ol F}$, $\varphi(W' _{\ol F})=W'_{0, \ol F}$, $\varphi(Y _{\ol F})=Y_{0,\ol F}$, and $\varphi(T(\ol F))=T_0(\ol F)$ and $\varphi(\mf t _{\ol F})=\mf t_{0,\ol F}$.

\begin{prop}\label{prop_unit_spc_det_1.1} 
Let $\gm_0 \in T_0(\ol F)$.  Then we have
\beq
\det {}_F(L_0)(\gm_0)= N ({\rm det}_{E} ((\gm_0)_{G_0}))^{\dim_{E} W} D_{H_{Y_0}}(\gm_0). 
\eeq
\end{prop}

\begin{proof}
Since $F$ is an infinite field, $T_{0,\reg}(F)$ is Zariski dense in $T_{0,\reg}(\ol F)$ by III 8.13 in \cite{Borel}.
The conclusion follows from Proposition \ref{prop_unit_spc_det}, because $\det_F(L)$ is a regular function on $T_{0,\reg}$.
\end{proof}

From Proposition \ref{prop_unit_spc_det_1.1}, by imitating the proof of Proposition \ref{prop_spo_spc_det}, we get the following result.
\begin{prop}\label{prop_unit_spc_det_1} 
For $\gm \in T_{\reg}(F)$ we have
\beq
\det {}_F L(\gm)=  N ({\rm det}_{E} (\gm_G))^{\dim_{E} W}  D_{H_{Y}}(\gm).
\eeq
\qed
\end{prop}

\begin{comment}
\begin{proof}  
We have
\beq
\begin{split}
\det {}_F L(\gm) &= \det{}_{\ol F} L(\gm) \\
			&= \det{}_{\ol F} ({}^\varphi L)({}^\varphi \gm)\\
			&=N ({\rm det}_{E} (({}^\varphi \gm)_{G_0})^{\dim_{E} W}) D_{H_{Y_0}}({}^\varphi \gm) \\
			&= N ({\rm det}_{E} (\gm_G)^{\dim_{E} W}) D_{H_{Y}}(\gm).
\end{split}
\eeq
The third equality follows from the previous proposition. 
\end{proof}
\end{comment}

\section{Final integration formulas} \label{LastSection}

\subsection{Haar measure}
Let $dn$ be a Haar measure on $N$. Since $N^{Y,T}$ is open in $N$ (by Corollary \ref{immersion}), we may restrict $dn$ to $N^{Y,T}$. By Proposition \ref{difffibre}, we obtain:
\begin{prop} 
Let $f \in L^1(N^{Y,T},dn)$.  Then $(f\circ\Sha_{Y,T}) \in L^1(M/\Delta_{T} \times T_{\reg}, \Sha_{Y,T}^{*}(dn))$ and 
\begin{equation*}
\int_{N^{Y,T}} f(n) dn=|W_{H_{Y}}(T)|^{-1}  \int_{M/\Delta_{T} \times T_{\reg}}  (f\circ\Sha_{Y,T})\Sha_{Y,T}^{*}(dn).
\end{equation*}
\qed
\end{prop}
Recall the set $N_{{\rm reg}}=\bigcup_{Y,T} N^{Y,T}$. 

\begin{prop}\label{ifone}
Let $f \in L^1(N,dn)$.  Then $f\circ \Sha_{Y,T} \in L^1(M/\Delta_{T} \times T_{\reg}, \Sha_{Y,T}^{*}(dn))$ for all $Y \in \mathcal Y_{k}$ (see Proposition \ref{scrsk}), and for all maximal tori $T$ of $H_{Y}$. Moreover,
\begin{equation*}
\int_{N} f(n) dn=\sum_{Y,T} |W_{H_{Y}}(T)|^{-1}  \int_{M/\Delta_{T} \times T_{\reg}}  (f\circ\Sha_{Y,T})\Sha_{Y,T}^{*}(dn).
\end{equation*}
The sum is taken over $Y \in \mathcal Y_{k}$ and conjugacy classes of maximal tori in $H_{Y}$.
\end{prop}

\begin{proof}
By Theorem \ref{Theorem 1}, the set $\bN_{{\rm reg}}$ of $\ol F$-points of $N_{{\rm reg}}$ is a nonempty Zariski subset of the affine space $\bN$. Since $N$ is an affine space, and $F$ is infinite, the set of $F$-points of $N_{\rm reg}$ is also nonempty. Moreover, it is the complement of a union of proper closed submanifolds of $N$, necessarily of smaller dimension than $N$. By Sard's Theorem (see \cite{BourbDiff}), the complement of $N_{{\rm reg}}$ in $N$ is of null measure. The measure $dn$ is determined by its restriction to the dense open set $N_{{\rm reg}}$. Since measures are determined locally (\cite{Bourbint} III, Section 2, Proposition 1), the measure $dn$ is the unique measure so that its restriction to each $N^{Y,T}$ is given by the above formula. This gives the proposition.
\end{proof}
  
\begin{prop} \label{unit_if}
Let $f \in L^1(N,dn)$ with $dn$ a Haar measure on $N$.  Then
\begin{equation*}
\int_N f(n) dn=\sum_{Y,T} |W_{H_{Y}}(T)|^{-1} \int_{T} |\delta_T(\gamma)| \int_{M/\Delta_{T}} f(\Int(m) n_Y(\gm))|\delta_N(m)| \frac{dm}{dz} d \gm
\end{equation*}
in the symplectic and orthogonal cases with $\dim W$ even and in the unitary case, and
\begin{equation*}
\int_N f(n) dn=\sum_{Y,T} |W_{H_{Y}}(T)|^{-1} \int_{T} |\delta_T(\gamma)| \int_{M/\Delta_{T}} f(\Int(m) n_Y(\epsilon_{Y}\gm))|\delta_N(m)| \frac{dm}{dz} d \gm
\end{equation*}
in the orthogonal case with $\dim W$ odd. Here
\begin{equation*}
 |\delta_T(\gm)|= \Bigg\{
\begin{array}{ll}
|\det \gm_G|^{\dim W + 1} | D_{H_{Y}}(\gm) |& \textrm{ in the symplectic case with $\dim W$ even},\\
|\det \gm_G|^{\dim W - 1} | D_{H_{Y}}(\gm) |&  \textrm{ in the orthogonal case with $\dim W$ even},\\
|\det (\epsilon\gm)_{G}|^{\dim W - 1} | D_{H_{Y}}(\gm) | & \textrm{in the orthogonal case with $\dim W$ odd},\\
|{\rm det}_{E} \gm_G|_E^{\dim_{E} W} |D_{H_{Y}}(\gm)|_F  & \textrm{ in the unitary case}.\\
\end{array}
\end{equation*}
\end{prop}
 
\begin{proof}
Follows from Propositions  \ref{deltas}, \ref{delta/det}, \ref{ifone}, \ref{prop_spo_spc_det}, and \ref{prop_unit_spc_det_1}.
\end{proof} 
 
\subsection{ $\Int(M)$-invariant version} \label{Ad-invt}

The presence of the factor $\delta_N$ in the above formula suggests that we replace $dn$ with an $\Int(M)$-invariant measure on $N$.   Such measures arise in the theory of intertwining operators developed by Goldberg and Shahidi (see for instance \cite{S92}, \cite{GS98}).

\begin{prop}
 For $n=n(\xi,\eta) \in N'$, put $\phi(n(\xi,\eta))=|\delta_N(m(\eta \ups_{0},1))|^{-\half}$.  Then 
\begin{enumerate}
\item For all $m \in M$ and $n \in N'$ we have $\phi(\Int(m)(n))=\phi(n) |\delta_N(m)|^{-1}$.
\item $d_Mn=\phi(n) dn$ is an $\Int(M)$-invariant measure on $N$.
\end{enumerate}
\qed
\end{prop}

One computes in the orthogonal case that  
\begin{equation*}
\delta_N(m(g,h))=(\det g)^{\dim W+ \dim X -1},
\end{equation*}
 in the symplectic case that,
\begin{equation*}
\delta_N(m(g,h))=(\det g)^{\dim W+ \dim X +1},
\end{equation*}
and in the unitary case that,
\begin{equation*}
\delta_N(m(g,h))=(N_{E/F}({\rm det}_{E} (g)))^{\dim_{E} W+ \dim_{E} X}.
\end{equation*}

(See Proposition 1 of \cite{IFS}.) 

\bigskip

Recall we have fixed an automorphism $\theta=\theta_Y$ of $G$ with $G^\theta$ isomorphic to $H_Y$.

\begin{defn}
For $x \in T_G$, we define
\begin{equation} \label{twist.disc.defn}
D^\theta_G(x)=\det(\Ad(x) \circ d\theta -1; \GG/ {}^\xi \mf t).
\end{equation}
\end{defn}

Note that by Proposition 12 in \cite{IFS} we have
\beq
{}^\xi \mf t= \{A \in \GG \mid \Ad(x) \circ d\theta (A)=A \}
\eeq
for $x=\gm_G$ with $\gm \in T_{\reg}$.

For the symplectic case and the orthogonal case with $W$ even-dimensional we further have
\begin{equation*}
\begin{split}
  |\det L(\gm)||\delta_N(\gm_G)|^{-\half} &=  |\det \gm_G|^{\pm \half+\half (\dim W-\dim X)}|D_{H_{Y}}(\gm)|\\
  								&=|D_{H_{Y}}(\gm)|^\half |D^\theta_G(\gm_G)|^\half |\det(\gm-1;Y)|^{\half \dim Y^{\perp}}, \\
  \end{split}
  \end{equation*}
  with the last equality following from Proposition 14 of \cite{IFS}.  In the first equality, the plus sign is chosen in the symplectic case, and the minus chosen in the orthogonal case.
  
  For the orthogonal case with $W$ odd-dimensional, we similarly have
  \begin{equation*}
  |\det L(\gm)||\delta_N((\epsilon_{Y}\gm)_G)|^{-\half} =| D_{H_{Y}}(\gm)|^\half |D^\theta_G((\epsilon_{Y}\gm)_G)|^\half |\det(\gm+1;Y)|^{\half \dim Y^{\perp}}. \\
  \end{equation*}
  
In the unitary case, we have
\beq
\begin{split}
  |\det L(\gm)|_{F}|\delta_N(\gm_G)|_{F}^{-\half} &=  |{\rm det}_{E} \gm_G|_{E}^{- \half (\dim_{E} W+\dim_{E} X)}|{\rm det}_{E} \gm_G|_{E}^{{\rm dim}_{E} W}|D_{H_{Y}}(\gm)|_{F}\\
  								&=|D_{H_{Y}}(\gm)|_{F} |{\rm det}_{E}(\gm-1;Y)|_{E}^{\half \dim_{E} Y^{\perp}} \\                                                                           
  \end{split}
\eeq
One computes in this case that $D_{H_{Y}}(\gm)=(-2)^{\dim_{E} W} D^\theta_G(\gm_G)$.

\subsection{A summary of the integration formulas}\label{finalform}

In this section we present the essentials of our result on the integration formulas more succinctly, for the convenience of the reader.  The ambient group $G^1$ is a (symplectic, orthogonal, or a unitary) group of isometries of a vector space $V$ with the corresponding nondegenerate sesquilinear form over a local field $F$ (or over a quadratic extension $E$, if it is unitary) with characteristic not equal to two.  A parabolic subgroup of $G^1$ corresponds to an isotropic subspace $W$ of dimension $k$.  Choosing a Levi subgroup $M$ leads to a decomposition $V=W+X+W'$, with $W,W'$ a hyperbolic pair, to which $X$ is orthogonal.  Write $H$ for the isometry group of $X$, and $G=\GL(W)$. We suppose that  $\dim W \leq \dim X$, and moreover that $\dim W$ is even in the symplectic case.  We choose linear isomorphisms of $W$ with nondegenerate subspaces $Y$ of $X$, up to permutation by $H$.  Given such a $Y$, we choose a certain involution $\theta=\theta_Y$ of $G$ with $G^\theta \cong \Isom(Y)$.  Also choose conjugacy classes of maximal tori $T$ of $\Isom(Y)$.    Let $d_Mn$ be the $\Int(M)$-invariant measure on $N$ prescribed above. On certain regular elements $\gm \in T$, of full measure in $T$, we have defined in this paper matching semisimple elements $\gm_G \in G$ and elements $n_Y(\gm) \in N$.  

\begin{defn} Let $\gm \in T_{\reg}$.
\begin{enumerate}
\item In the symplectic case and the orthogonal case with $\dim W$ even, put
\beq
J_{T}(\gm)=|D_{H_{Y}}(\gm)|^\half |D^\theta_G(\gm_G)|^\half |\det(\gm-1;Y)|^{\half \dim Y^{\perp}}.
\eeq
\item In the orthogonal case with $\dim W$ odd, put 
\beq
J_{T}(\gm)=|D_{H_{Y}}(\gm)|^\half |D^\theta_G((\epsilon_Y\gm)_G)|^\half |\det(\gm+1;Y)|^{\half \dim Y^{\perp}}.
\eeq
\item In the unitary case, put
\beq
 J_{T}(\gm)= |D_{H_{Y}}(\gm)|_{F} ^\half |D^\theta_G(\gm_G)|_{F}^\half |{\rm det}_{E}(\gm-1;Y)|_{E}^{ \half \dim_{E} Y^{\perp}}.
 \eeq
 \end{enumerate}
\end{defn}

\begin{remark}
Of course $\dim Y^{\perp}=\dim X-\dim Y=m-k$.
\end{remark}

\begin{thm} \label{intro_thm_1}
Let $f \in L^1(N,d_Mn)$. Then up to normalization of measure we have:
\begin{equation*}
\int_N f(n) d_Mn=\sum_{Y,T}|W_{H_{Y}}(T)|^{-1}\int_{T} J_{T}(\gm)\int_{M/{\Delta_T}}f(\Int(m)n_{Y}(\gamma))\frac{dm}{dz}d\gamma,
\end{equation*}
except in the odd orthogonal case, where we have:
\begin{equation*}
\int_N f(n) d_Mn=\sum_{Y,T}|W_{H_{Y}}(T)|^{-1}\int_{T} J_{T}(\gm)\int_{M/{\Delta_T}}f(\Int(m)n_{Y}(\epsilon_Y\gamma))\frac{dm}{dz}d\gamma.
\end{equation*}
The sum is taken over $H$-orbits of nondegenerate subspaces $Y$ of $X$ of dimension $k$, and conjugacy classes of maximal tori $T$ in  $H_Y=\Isom(Y)$.
\qed
\end{thm}
Let us work out  a couple of small examples, using the notation of Section \ref{MatrixSection}.
 
\subsubsection{Symplectic example}

Let $G^1={\rm Sp}_{8}(F)$, viewed as the isometry group of the form $J_-(8)$.  Let $W$ be $2$-dimensional, so that $V=W+X+W'$ with a $4$-dimensional $X$.
The set $\mc Y$ equals $\{Y \}$ for a $2$-dimensional nondegenerate subspace $Y$ of $X$.  With the choice of $\xi$ as in Section \ref{MatrixSection_1}, one computes $\theta(g)=\dfrac{g}{\det (g)}$.

Of course $H_Y\cong\SL_2(F)$, and the maximal tori $T < H_Y$ are either split or the groups of their $F$-rational points correspond to norm-one elements of quadratic extensions $E$ of $F$.  Suppose for simplicity that $T$ is split diagonal.  If $\gm \in T_{\reg}$, then $\gm=\left(\begin{array}{cc}
a & 0 \\
0 & a^{-1} \\
\end{array} \right)_Y$ for some $a \in F^\times$, with $a \neq \pm 1$.  Then $\gm_G=\left(\begin{array}{cc}
a-1 & 0 \\
0 & a^{-1}-1 \\
\end{array} \right)^{-1}_W$.

 We have $|\det(\gm_G)|=|a-1|^{-1}|a^{-1}-1|^{-1}$, $|D_{H_Y}(\gm)|=|a^2-1||a^{-2}-1|$, so
 \beq
  J_{T}(\gm)=|a-1|^{3}|a+1|^{2} |a|^{-5/2}.
 \eeq
\subsubsection{Orthogonal example}
Now let $G^1={\rm O}_{5}(F)$, viewed as the isometry group of the form $J_V$ described in Section \ref{ping}.  Let $W$, $W'$ and $X$, with $\dim W=1$ and $\dim X=3$, be such that $V=W+X+W'$ as above. The set $\mc Y$ in this case consists of the $H$-orbits of the $1$-dimensional anisotropic subspaces of $X$, which are parameterized by the set $F^{\times}/(F^{\times})^{2}$. Let $\{Y(a)\}$ be a fixed set of representatives of the $H$-orbits, defining elements of $\mc Y$ (where $a\in F^{\times}/(F^{\times})^{2}$). Note that $G={\rm GL}_{1}(F)$ in this case, and $\theta(g)=g^{-1}$. Moreover $H_{Y(a)}=\{\pm 1\}$ and $T=\{1\}$. We compute $(\epsilon\gm)_G=(-2^{-1})_{W}$, $|\det((\epsilon_{Y(a)}\gm)_G)|=|2|^{-1}$, $|D_{H_{Y(a)}}(\gm)|=1$, and $|D^\theta_G((\epsilon_{Y(a)}\gm)_{G})|=|2|$. Thus $J_{T}(\gm)=|2|^{\frac{3}{2}}$. For $f \in L^1(N,d_Mn)$, we have (up to normalization of measures):
\begin{equation*}
\int_{N}f(n)d_{M}n=\sum_{a\in F^{\times}/(F^{\times})^{2}}\int_{M}f(\Int(m)n_{Y(a)}(\epsilon_{Y(a)}))dm.
\end{equation*}

\section{A Goldberg-Shahidi pairing}  \label{Residues.Section}
\subsection{Induced representations}
The integration formulas obtained in this article are central to an ongoing project to determine residues of intertwining operators acting on parabolically induced representations of classical groups over $p$-adic fields.  In this section we display the expected analogue of the main result of \cite{Spallone}, which was restricted to the case of $\dim W=\dim X$.  We omit the details, but Lemma \ref{ralph} above is the main entrypoint. Now we assume that $F$ is a $p$-adic field. Let $q$ denote the cardinality of the residue field of $F$.  

Let $(\pi_{G}, V_G)$ and $(\pi_{H},V_H)$ be unitarizable supercuspidal representations of $G$ and $H$ respectively.  We assume that $\pi_{G}$ is self-dual in the orthogonal and symplectic cases, and conjugate self-dual in the unitary case.  Write $Z$ for the center of $G$.  Let $\omega$ denote the central character of $\pi_{G}$.  Define a representation $(\pi_M,V_M)$ of $M$ via $\pi_M=\pi_G \boxtimes \pi_H$, the external direct product. Let $\upsilon_{0}:W\to W'$ be a fixed self-adjoint isomorphism, and use this to define $w_0$ as in Section \ref{w_0.here}.  Let $\alpha$ be the simple root in $N$ and $\tilde{\alpha}$ the corresponding fundamental weight. For each $s \in \C$ we have the induced representation
\begin{equation*}
V(s\tilde{\alpha},\pi_M)={\rm Ind}^{G^1}_P \pi_M \otimes q^{\langle s\tilde{\alpha}, H_{M}()\rangle}\otimes \mathbf{1}
\end{equation*}
of $G^1$, where $H_{M}$ is Harish-Chandra's height function.  Let $A(s\tilde{\alpha},\pi_M,w_{0})$ be the usual intertwining operator (\ref{int.op}) for $w_{0}$.

By the theory developed by Harish-Chandra (see \cite{Sil}), the problem of determining the reducibility points of $V(s\tilde{\alpha},\pi_M)$ amounts to determining the poles of this intertwining operator, as a function of $s$. Moreover, Shahidi's $L$-functions are defined in terms of these poles (see \cite{ShahidiBook}).

\subsection{Definition of the pairing}

Let $f_H$ be a matrix coefficient of $\pi_H$, and $f_G$ a matrix coefficient of $\pi_G$.  
 Recall that for each $Y$, we have chosen $\xi_Y$ and so have $\ups_{Y}=(\xi_{Y}\xi_{Y}^*)^{-1}: W \to W'$.  We also use $\xi_Y$ to identify $H_Y$ with a subgroup of $G$.   
 Also let $x_{Y}=\upsilon_{Y}^{-1}\upsilon_{0}$.

\begin{defn} Let $\phi_H \in C_c^{\infty}(H)$ and $\phi_G \in C^{\infty}(G)$, with $\phi_G$ compact modulo $Z$.  Write $I^H_\gm(\phi_H)= |D_{H_{Y}}(\gm)|^\half \int_{H/H_\gm}  \phi_{H}(h \gm h^{-1}) d \dot{h}$.  Put
\beq
I^G_{ \gm_G}(\phi_G)= |D^\theta_G(\gm_G)|^\half \int_{G/Z_{G}Z_{H_{Y}(T)}}  ({}^{x_{Y}}\phi_{G})(g\gamma_{G}\theta(g)^{-1}) d\dot{g},
\eeq
except in the orthogonal case with $\dim W$ odd, in which case put
\beq
I^G_{ \gm_G}(\phi_G)= |D^\theta_G((\epsilon_Y \gm)_G)|^\half \int_{G/Z_{G}Z_{H_{Y}(T)}}  ({}^{x_{Y}}\phi_{G})(g (\epsilon_Y \gamma)_{G}\theta(g)^{-1}) d\dot{g}.
\eeq

\end{defn}

Here  $({}^{x_Y}\phi)(g)=\phi(x_Y^{-1}g)$)

These are the appropriate normalized orbital integrals in this context.

\begin{defn} Given $\gm \in T_r$, put 
\beq
j_T(\gm)=|\det(\gm \pm 1;Y)|^{\half \dim Y^{\perp}},
\eeq
with the sign being `$-$' in all cases except orthogonal with $\dim W$ odd, in which the sign is `$+$'.  In the unitary case we mean the $E$-absolute value, the $E$-determinant, and dimension as an $E$-space.
\end{defn}
 
Thus $j_T(\gm)$ is simply $J_T(\gm)$ divided by the discriminant factors.

\begin{thm}\label{th_exp}  (Expected)\
The intertwining operator has a pole at $s=0$ if and only if there exist matrix coefficients $f_G$  of $\pi_G$ and $f_H$ of $\pi_H$, so that $\GS(f_G,f_H)\neq 0$, where
\beq
\GS(f_G,f_H) =  \sum_Y \sum_{ T_c \leq H_Y} |W_{H_Y}(T_c)|^{-1}\int_{T_c} j_{T_c}(\gm) I_{\gm}^H(f_H) I_{\gm_G}^G(f_G) d \gm. 			 
\eeq
Here $Y$ runs over $H$-orbits of nondegenerate subspaces of $X$ of dimension $k$, and $T_c$ runs over conjugacy classes of {\it compact} maximal tori in $H_{Y}$.
\end{thm}
 
Theorem \ref{th_exp} when available will place one in a position to apply techniques of harmonic analysis, for instance the Selberg principle.

\subsection{Current and future examples}\label{sosix}
 
 Let $E$ be a quadratic extension of $F$.  Take $V$ of the form $V=W+X+W'$ with $W$ two-dimensional and $X$ isomorphic to $E$ with quadratic form given by $N_{E/F}$.
 Put $G^1=\SO(V)$, i.e.,  the quasisplit even orthogonal group of rank $3$ determined by $E$.  It has a Levi subgroup $M \cong \GL_2 \times \SO(X)$.  Here $\SO(X)$ is isomorphic to $E^{1}$, the subgroup of norm $1$ elements in $E^{\times}$.  So in this case $H$ is its own unique maximal torus.
 
In \cite{Comp}, $\GS(f_{G},f_{H})$ was computed explicitly, by using the Selberg principle and an endoscopic transfer identity of Labesse and Langlands \cite{LL}. Let $\pi_{H}$ be a character $\chi$ of $E^{1}$. For simplicity, let us assume that the central character $\omega$ of $\pi_G$ is nontrivial. In this case, the Langlands parameter of $\pi_{G}$ is of the form ${\rm Ind}^{W_{F}}_{W_{E'}}(\chi')$ where $E'$ is a quadratic extension of $F$ and $\chi'$ is a character of $(E')^{\times}$. The determination of when $\GS(f_{G},f_{H})$ was nonzero gave the following reducibility criterion, in accordance with Shahidi \cite{Sha.Appendix}:
\begin{thm}
 $V(0,\pi_M)$ is reducible unless $E=E'$ and the Langlands parameter of $\pi_{G}$ is ${\rm Ind}^{W_{F}}_{W_{E}}(\chi)$. If the Langlands parameter of $\pi_{G}$ is ${\rm Ind}^{W_{F}}_{W_{E}}(\chi)$, then $V(s\tilde{\alpha},\pi_M)$ is reducible at $s=1$ and there are no other points of reducibility for $s \geq 0$.
 \end{thm}

In \cite{WWL}, Wen-Wei Li applied the ``equal-sized'' integration formula from \cite{IFS}, and Waldspurger's formula for transfer factors \cite{Waldspurger} in generalizing this to a similar statement for even orthogonal groups.
\begin{comment}
 
\begin{thm} 
Assume that $\pi_G$ is not obtained by transfer from $\SO_{2n+1}$.  Then $\GS(f_G,f_H)$ is not identically zero if and only if $\pi_G$ comes from $H$ and the $\Out(H)$-orbit of $\pi_H$ is contained in $\Pi_\phi$.  
\end{thm}
(See \cite{WWL} for the meaning of $\Out(H)$ and $\Pi_\phi$.)
 \end{comment} 
  
We expect to apply the results of this paper to obtain analogous results in the case of $G=\GL_{2m}$ and $H=\SO_{2m+1}$, which was emphasized by Goldberg and Shahidi in \cite{GSIII}.
In this case, the nonvanishing of the residue formula signifies that $\pi_G$ is the local transfer of $\pi_H$ (see Theorem 5.4 of \cite{GSIII} for details).  

Another clear direction is to apply the recent work of Mok \cite{Mo} to obtain such results for the quasisplit unitary groups $U(3n,3n)$.  The work \cite{Cai-Xu} of Li Cai and Bin Xu deals with the first nontrivial case of $G=\GL_2(E)$ and $H=U(1,1)$. There are two standard base change maps, the stable base change map and the unstable base change map, which take irreducible representations of $U(n,n)$ to those of $\GL_{2n}(E)$ (see Section 2 of \cite{Mo} for the definitions). As before let $\pi_{G}$ and $\pi_{H}$ be unitarizable supercuspidal representations of $\GL_{2n}(E)$ and $U(n,n)$ respectively, and $f_{G}$ and $f_{H}$ be matrix coefficients of the two representations. Further assume that $\pi_{G}$ lies in the image of the stable base change map. We expect to show in a later work that there exist $f_G,f_H$ so that $\GS(f_G,f_H) \neq 0$   precisely when $\pi_H$ is taken to $\pi_{G}$ by the stable base change map, thus generalizing the main result of \cite{Cai-Xu}. By the results of Harish-Chandra mentioned in the beginning of this section and Theorem \ref{th_exp}, it will follow that the induced representation $V(0,\pi_G \boxtimes \pi_{H})$ is irreducible if and only if $\pi_{G}$ is the stable base change lift of $\pi_{H}$.

\def\cprime{$'$} \def\cprime{$'$} \def\Dbar{\leavevmode\lower.6ex\hbox to
  0pt{\hskip-.23ex \accent"16\hss}D}
  \def\cftil#1{\ifmmode\setbox7\hbox{$\accent"5E#1$}\else
  \setbox7\hbox{\accent"5E#1}\penalty 10000\relax\fi\raise 1\ht7
  \hbox{\lower1.15ex\hbox to 1\wd7{\hss\accent"7E\hss}}\penalty 10000
  \hskip-1\wd7\penalty 10000\box7}
  \def\cfudot#1{\ifmmode\setbox7\hbox{$\accent"5E#1$}\else
  \setbox7\hbox{\accent"5E#1}\penalty 10000\relax\fi\raise 1\ht7
  \hbox{\raise.1ex\hbox to 1\wd7{\hss.\hss}}\penalty 10000 \hskip-1\wd7\penalty
  10000\box7} \def\cftil#1{\ifmmode\setbox7\hbox{$\accent"5E#1$}\else
  \setbox7\hbox{\accent"5E#1}\penalty 10000\relax\fi\raise 1\ht7
  \hbox{\lower1.15ex\hbox to 1\wd7{\hss\accent"7E\hss}}\penalty 10000
  \hskip-1\wd7\penalty 10000\box7} \def\cprime{$'$}
  \def\Dbar{\leavevmode\lower.6ex\hbox to 0pt{\hskip-.23ex \accent"16\hss}D}
  \def\cftil#1{\ifmmode\setbox7\hbox{$\accent"5E#1$}\else
  \setbox7\hbox{\accent"5E#1}\penalty 10000\relax\fi\raise 1\ht7
  \hbox{\lower1.15ex\hbox to 1\wd7{\hss\accent"7E\hss}}\penalty 10000
  \hskip-1\wd7\penalty 10000\box7}
  \def\polhk#1{\setbox0=\hbox{#1}{\ooalign{\hidewidth
  \lower1.5ex\hbox{`}\hidewidth\crcr\unhbox0}}} \def\dbar{\leavevmode\hbox to
  0pt{\hskip.2ex \accent"16\hss}d}
  \def\cfac#1{\ifmmode\setbox7\hbox{$\accent"5E#1$}\else
  \setbox7\hbox{\accent"5E#1}\penalty 10000\relax\fi\raise 1\ht7
  \hbox{\lower1.15ex\hbox to 1\wd7{\hss\accent"13\hss}}\penalty 10000
  \hskip-1\wd7\penalty 10000\box7}
  \def\ocirc#1{\ifmmode\setbox0=\hbox{$#1$}\dimen0=\ht0 \advance\dimen0
  by1pt\rlap{\hbox to\wd0{\hss\raise\dimen0
  \hbox{\hskip.2em$\scriptscriptstyle\circ$}\hss}}#1\else {\accent"17 #1}\fi}
  \def\bud{$''$} \def\cfudot#1{\ifmmode\setbox7\hbox{$\accent"5E#1$}\else
  \setbox7\hbox{\accent"5E#1}\penalty 10000\relax\fi\raise 1\ht7
  \hbox{\raise.1ex\hbox to 1\wd7{\hss.\hss}}\penalty 10000 \hskip-1\wd7\penalty
  10000\box7} \def\lfhook#1{\setbox0=\hbox{#1}{\ooalign{\hidewidth
  \lower1.5ex\hbox{'}\hidewidth\crcr\unhbox0}}}
\providecommand{\bysame}{\leavevmode\hbox to3em{\hrulefill}\thinspace}
\providecommand{\MR}{\relax\ifhmode\unskip\space\fi MR }
\providecommand{\MRhref}[2]{%
  \href{http://www.ams.org/mathscinet-getitem?mr=#1}{#2}
}
\providecommand{\href}[2]{#2}

\end{document}